\documentclass[reqno]{amsart}            
\usepackage{graphicx,cite,color}
\usepackage[bookmarks,bookmarksnumbered,bookmarksopen,colorlinks,backref,linkcolor=blue,citecolor=red]{hyperref}%
\usepackage{xcolor}
\usepackage{dsfont}
\usepackage{mathrsfs}
\usepackage{amsmath, amssymb ,amsthm, amsfonts, amsgen}

\newtheorem{theorem}{Theorem}[section]
\newtheorem{definition}{Definition}[section]
\newtheorem{lemma}{Lemma}[section]
\newtheorem{corollary}{Corollary}[section]
\newtheorem{remark}{Remark}[section]
\numberwithin{equation}{section}
\newtheorem{proposition}{Proposition}[section]

\begin{document}

\title[Inner-layer asymptotics for coupling in  partially perforated domains]
{Inner-layer asymptotics in partially perforated domains: coupling across flat and oscillating interfaces}
\author[Taras Mel'nyk]{ Taras Mel'nyk$^{\natural, \flat}$}
\address{\hskip-12pt
$^\natural$ Institute of Applied Analysis and Numerical Simulation,
Faculty of Mathematics and Physics,  University of Stuttgart\\
Pfaffenwaldring 57,\ 70569 Stuttgart,  \ Germany
 \newline
$^\flat$ Faculty of Mathematics and Mechanics\\
Taras Shevchenko National University of Kyiv\\
Volodymyrska str. 64,\ 01601 Kyiv,  \ Ukraine\\
}
\email{Taras.Melnyk@mathematik.uni-stuttgart.de}

\begin{abstract} \vskip-10pt
The article examines a boundary-value problem in a domain consisting of  perforated and imperforate regions, with Neumann conditions prescribed at
 the boundaries of the perforations. Assuming the porous medium has  symmetric, periodic structure  with a small period $\varepsilon,$  we analyse the limit behavior of the problem as $\varepsilon \to 0.$  A crucial aspect of this study is deriving  correct coupling conditions at the common interface, which is achieved using inner-layer asymptotics.
For the flat interface, we construct and justify a complete asymptotic expansion of the solution in the $H^1$-Sobolev space.
Furthermore, for the $\varepsilon$-periodically oscillating interface of amplitude $\mathcal{O}(\varepsilon),$  we provide an approximation to the solution and establish the corresponding asymptotic estimates in $H^1$-Sobolev spaces.
\end{abstract}

\keywords{Asymptotic approximation, partially perforated domains, inner-layer asymptotics, oscillating interface
\\
\hspace*{9pt} {\it MOS subject classification:} \   35B27,  35B40, 35B25,  35J25
}

\maketitle
\tableofcontents
\section{Introduction}\label{Sect1}

Coupled systems with distinct structures play a crucial role across various scientific disciplines. In recent years, extensive research has focused on exploring different models within coupled media of varying structures (see, e.g., \cite{Egg-Ryb-2021,Kro-Ola-Ryb-2023,Str-Egg-Ryb-2023,Wei-Koc-Hel-2022,Helmig-2020,Roh-Ryb-2016}). A key challenge in these studies is determining the transmission (coupling) conditions at the interface between them. To address this, researchers have proposed a range of generalized interface conditions. However, their justification has been carried out with varying degrees of rigor, primarily through numerical analysis and comparisons with classical transmission conditions. Additionally, asymptotic approaches have been developed to examine models within coupled media exhibiting different periodic microstructures.

\paragraph*{{\bf Coupling across flat interface}}
To the best of my knowledge, the paper \cite{Pan-1981} was the first to introduce a general methodology for homogenizing boundary value problems involving contact between two periodic inhomogeneous half-spaces with a flat interface. The author examined both scenarios: one in which the two media are separated by a thin inhomogeneous layer with a periodic structure, and another where they are in direct contact. In this approach, internal boundary layers were considered only in the presence of a thin inhomogeneous separating layer. For the case of direct contact between two periodic inhomogeneous half-spaces, only general observations were provided in \S 7.

The first asymptotic results for a bounded partially perforated domain with a flat interface were established in \cite{Jag-Ole-Sham-1993} for the Poisson equation, assuming zero Neumann conditions on the perforation boundaries. However, the proposed approximation to the solution is discontinuous at the interface.  Moreover, asymptotic estimates for the difference between the original solution and its approximation were provided separately for the perforated and non-perforated subdomains. These estimates are of order $\mathcal{O}(\sqrt{\varepsilon})$ (see Theorem 2), where $\varepsilon$ represents both the perforation period and the characteristic size of the hole diameters.

In \cite{Ole-Shap-1995}, the authors demonstrated that when the diameter of the holes is asymptotically smaller than the perforation period, the holes have no effect on the leading term of the asymptotics of the solution to the Poisson equation in a partially perforated domain with zero Neumann boundary conditions on the hole boundaries. Additionally, they derived estimates for the difference between the solutions of the initial and limit problems in the $H^1$-Sobolev norm.

A boundary-value for the Poisson equation in a partially perforated domain  with Robin conditions $\partial_{\boldsymbol{\nu}_\varepsilon}u_\varepsilon + \varepsilon^k u_\varepsilon =0$ $(k\in \Bbb R)$ at the cavity boundaries  was studied in  \cite{Ole-Shap-1995-Robin}. Three cases were discovered in the asymptotic behaviour of the solution $u_\varepsilon$: $k <1,$ $k=1$ and $k>1.$  In each case, a corresponding homogenised problem was derived and asymptotic estimates were obtained in the subdomains as in \cite{Jag-Ole-Sham-1993}.

In the case of zero Dirichlet conditions on the boundaries of the holes, the first term of the asymptotics is a solution of the Poisson equation in a non-perforated subdomain with zero Dirichlet condition at the interface \cite{Ole-Sham-1994}. In \cite{Jag-Mik-1996}, the conjugation conditions for the Stokes system in a partially perforated infinite strip with the Dirichlet condition on the boundaries of the holes were derived. The authors constructed correctors for the pressure and velocity and proved $L^2$-estimates for them.

In \cite{Pan-1999}, a linear stationary problem of the thermal field in an infinite strip was examined, consisting of a highly conductive infinite sub-strip and a periodically perforated infinite sub-strip. Dirichlet conditions with a small perforation period $\varepsilon$ were imposed at the perforation boundary, while the thermal conductivity coefficient in the highly conductive region was represented by a large parameter $\omega$. An asymptotic expansion was constructed as $\omega \to \infty$ and $\varepsilon \to 0$, incorporating boundary layers. However, the partial sums proposed for justification lack continuity at the interface. The study also considered Neumann conditions on the perforations, and Theorem 3 provided $L^2$-estimates for the difference between the solution of the original problem and those of the corresponding limit problems in the sub-strips.
It should be noted that the consideration of the problem in an infinite strip eliminates the need for additional boundary layer constructions, significantly simplifying the analysis.

\paragraph*{{\bf Coupling across oscillating interface}}
It is often the case that  interfaces in coupled media exhibit rapidly oscillating structures. Understanding the impact of such complex interfaces on the dynamics of coupled systems is crucial. Numerous studies have focused on the homogenization of boundary value problems in domains composed of two heterogeneous media separated by rapidly oscillating interfaces with varying amplitudes (see, e.g., \cite{Don-Piat-2010,Don-Gia-2016,Don-2019,Don-Pet-2022} and the references therein). The following  imperfect contact
 transmission conditions were considered in these papers: the continuity of the flux and the proportionality of the flux  to the solution's jump at the interface. In addition,  the amplitude of the interface oscillations is of order $\varepsilon^\kappa,$ where $\kappa \ge 0,$ and the
 proportionality coefficient appearing in the transmission conditions is of order $\varepsilon^\gamma,$ with $\gamma \in \Bbb R.$
It is interesting to note that for the main values of the parameters $\kappa > 0$ and $\gamma$ there is no influence of the interface microstructure in the corresponding homogenized problem.  This influence is present  when  $\kappa =0 $ (see \cite{Don-Pet-2022} and
\cite{Gaudiello-1995,Nev-Kel-1997} for the classical transmission conditions when the solution and the normal flux are continuous across the interface).

\smallskip

The present paper examines a boundary-value problem for the Poisson equation in a bounded, partially perforated domain, where Neumann conditions are imposed on the boundaries of the perforations. The perforated region exhibits a symmetric, periodic structure with a small period $\varepsilon.$ Two distinct configurations for the separation of these regions are considered: one featuring a flat interface and the other incorporating an $\varepsilon$-periodically oscillating interface with an amplitude of order $\mathcal{O}(\varepsilon).$ The classical transmission conditions are prescribed at the interface, with one exception (see below).

In the first case, we construct a complete asymptotic expansion in the entire partially perforated domain and establish the corresponding asymptotic estimates for the solution in the Sobolev space. The construction of a continuous asymptotic approximation with arbitrary accuracy within the entire bounded, partially perforated domain has long remained an open problem. Only one additional assumption is required: the symmetry of the perforation cell.
Symmetry plays a crucial and intriguing role in nature. For example, symmetry is fundamental to the formation of crystals; in chemistry, molecules with symmetrical shapes tend to be more stable; and in biology, many proteins exhibit symmetrical structures essential to their function. In this paper, we demonstrate how the symmetry of the perforation cell facilitates the construction of the asymptotic expansion for the solution.

An asymptotic expansion provides detailed information about the structure of the solution, which is essential for the accurate modeling of complex physical phenomena. In our case, we establish approximations for both the solution and its gradient across the entire partially perforated domain
with arbitrary accuracy and rigorously justify higher-order transmission conditions (see Theorem~\ref{Theorem-main}).

In the second case (oscillating interface), it was possible to construct a two-term asymptotic approximation and prove the corresponding estimates in Sobolev spaces in the perforated and non-perforated regions. This case is characterized by the fact that the influence of the interface microstructure does not manifest itself  in the homogenized problem, which is consistent with the results obtained in \cite{Don-Piat-2010,Don-2019} for 
imperfect contact  conditions. However, only by using inner-layer asymptotics, this influence can be detected and identified in the second terms of the asymptotics, as demonstrated in this paper.

It has also been shown that if the second transmission condition is
\begin{equation*}
  D^- \nabla_x u^-_{\varepsilon} \cdot \vec{\nu}_\varepsilon = \nabla_x u^+_{\varepsilon} \cdot \vec{\nu}_\varepsilon  + \Theta(x_2, \tfrac{x_2}{\varepsilon})
\ \ \text{on} \ \ \left\{x\colon \ x_1= \varepsilon \,  \ell(\tfrac{x_2}{\varepsilon}), \  x_2 \in (0, d)\right\},
\end{equation*}
where $\Theta(x_2, \xi_2), \ x_2\in [0, d], \, \xi_2 \in [0,1],$ is $1$-periodic  in $\xi_2$ and smooth given function, then the corresponding second conjugation condition in the homogenized problem is as follows
$$
D^-\, \partial_{x_1}v^-_0(x)\big|_{x_1 =0} =\  \upharpoonleft\!\! Y\!\!\upharpoonright  {h_{11}} \, \partial_{x_1}v^+_0(x)\big|_{x_1 =0} + \widehat{\Theta}(x_2), \quad x_2 \in (0,d),
$$
where the function $\widehat{\Theta}(x_2) = \int_{0}^{1}\Theta(x_2, \xi_2) \, \sqrt{1+ |\ell^\prime(\xi_2)|^2} \, d\xi_2$ exhibits
the impact of the interface microstructure. For more details, see Sect.~\ref{Sect-6}.

The paper has the following structure. The precise formulation of the problem in the first case  is presented  in Sect.~\ref{Sec-2}.
 Section~\ref{Sect_3} is devoted to the construction of a formal asymptotic expansion of the solution.
 The asymptotic expansion consists of three parts:   a power series in degrees of the parameter $\varepsilon$ in the non-perforated region,
 a standard two-scale asymptotic ansatz in the perforated region, and an inner-layer  series  in a vicinity of the interface between the perforated and non-perforated parts. Here, the solvability of all interconnected recurrent procedures that determine the coefficients of these series  is proven, paying more attention to the inner-layer asymptotics and  the influence of the symmetry of the periodicity cell on the solutions of these problems.
 In Sect.~\ref{Sect-4}, using these series, we construct a series in the entire partially perforated domain $\Omega_\varepsilon$, prove that it is the asymptotic expansion for the solution in the Sobolev space $H^1(\Omega_\varepsilon).$  The case of the rapidly oscillating interface is studied in Sect.~\ref{Sect-5}.
The article ends with a section of conclusions and remarks.

\section{The problem statement}\label{Sec-2}

For clarity and  brevity, the problem is considered in $\Bbb R^2.$ Of course, this approach does not depend on the dimensionality of the space.
Let $G_0$ be a finite union of smooth disjoint nontangent domains strictly lying in the unite square
$\square := \{\xi=(\xi_1, \xi_2) \in \Bbb R^2\colon \ 0 <  \xi_1 <  1, \ 0 <  \xi_2 < 1\}.$ Denote by
$
Y := \overline{\square}  \setminus \overline{G_0}
$
(see Fig.~\ref{f1}).
The main our assumptions is the symmetry of $Y$  with respect the lines
$\{\xi\colon \ \xi_1 = \frac{1}{2}\}$ and $\{\xi\colon \ \xi_2 = \frac{1}{2}\}.$
\begin{figure}[htbp]
  \centering
  \includegraphics[width=4.5cm]{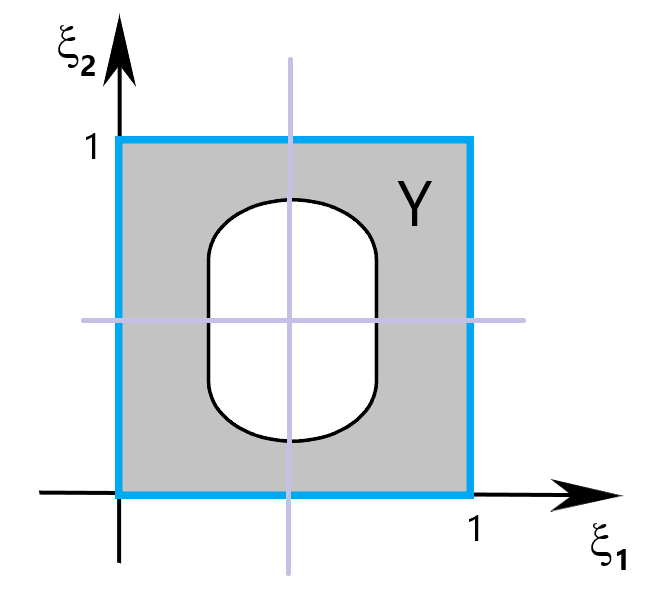}
  \caption{The periodicity cell  $Y$}\label{f1}
   \end{figure}

Let $\Omega^+$ be a square $\{ x=(x_1, x_2) \in \Bbb R^2\colon \ 0 < x_1 < d, \ \ 0 < x_2 < d\},$ and
$\Omega^-$ be a bounded domain in $\Bbb R^2$ lying in the left half-plane $\{ x \in \Bbb R^2\colon \  x_1 < 1\}.$ It is also assumed that there exists a small positive number $\varrho_0$ such that
$$
\Omega^- \cap \{ x\colon \  - \varrho_0 < x_1 < 0\} = \{ x\colon \  - \varrho_0 < x_1 < 0, \ \ 0< x_2 < d\}.
$$
We assume that $\Gamma^- := \partial\Omega^- \setminus \mathcal{Z}$ is a smooth curve, where $\mathcal{Z}:= \{x\colon x_1=0, \ \ 0< x<d\}.$
Denote by
$$
\Omega := \Omega^- \cup \mathcal{Z} \cup \Omega^+.
$$

Let  $\varepsilon = \dfrac{d}{N},$ where $N$ is a large positive integer,  and
$$
 \mathcal{Y}_{{\varepsilon}} := \bigcup_{k,n \, \in \, \Bbb Z} \big( {{\varepsilon}}Y + {\varepsilon} (k, n)\big),
$$
where $\varepsilon Y $ is the homothety of $Y$ with coefficient $\varepsilon.$

\begin{figure}[htbp]
  \centering
  \includegraphics[width=7cm]{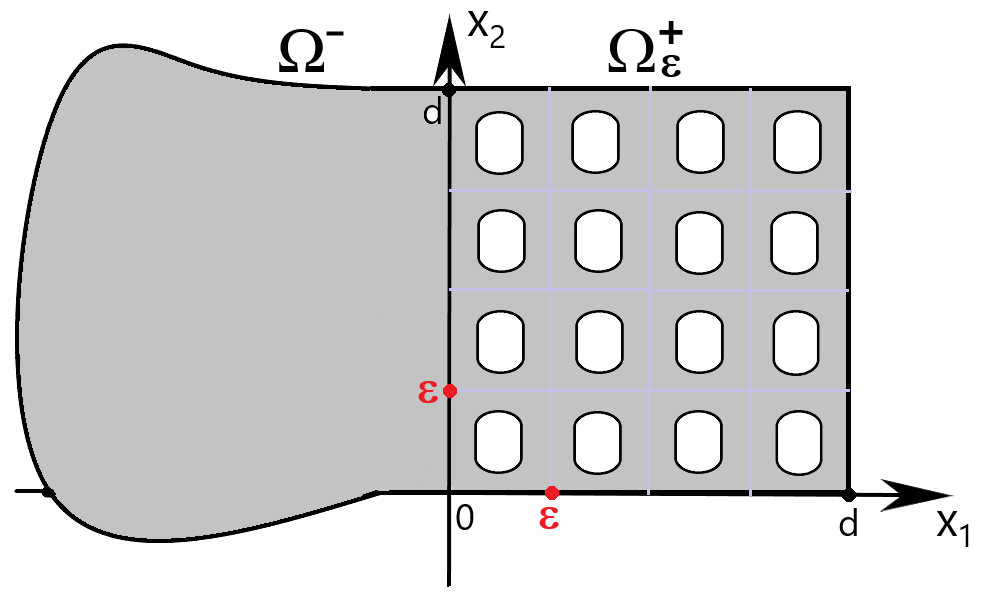}
  \caption{The partial perforated domain  $\Omega_\varepsilon$}\label{f2}
   \end{figure}

Then the perforated square  $\Omega^+_\varepsilon:= \Omega^+ \cap  \mathcal{Y}_{{\varepsilon}}$ and
partially perforated domain
$
\Omega_\varepsilon := \Omega^-   \cup {\mathcal{Z}} \cup \Omega^+_{{\varepsilon}}
$
(see Fig.~\ref{f2}).
In the paper, the $-$ index at the top will always indicate the connection to the left part of the domain, and $+$ to the right one.

In $\Omega_\varepsilon$ we consider the problem
 \begin{equation}\label{original Problem}
\left\{
\begin{array}{rcll}\displaystyle
\Delta_{ x}u_{\varepsilon}(x) &=& f(x),  & \quad x \in  \Omega_\varepsilon,
\\[2pt]
\partial_{\vec{\nu}_\varepsilon}u_{\varepsilon}(x) &=& 0,  & \quad x \in \partial G_{\varepsilon}  \ \ (\text{the boundaries of the holes}),
\\[2pt]
u_{\varepsilon}(x) &=& 0,& \quad x  \in \partial\Omega,
\end{array}
\right.
\end{equation}
where \ $ \Delta_{x} := \partial^2_{x_1} + \partial^2_{x_2 },$ \
$\partial_{x_i} = \frac{\partial}{\partial x_i},$ \ $\partial^2_{x_i} = \frac{\partial^2}{\partial x_i \partial x_i},$
$\partial_{\vec{\nu}_\varepsilon}u := \nabla_x u \cdot \vec{\nu}_\varepsilon,$
$\vec{\nu}_\varepsilon = \big(\nu_1(\tfrac{x}{\varepsilon}), \, \nu_2(\tfrac{x}{\varepsilon}) \big)$ is the unit  normal to $\partial G_{\varepsilon},$
external to $\Omega^+_\varepsilon.$

To construct an asymptotic expansion for the solution, the problem data must be infinitely smooth functions.
This requirement is due to the fact that the expansion coefficients are determined in terms of the derivatives of the preceding coefficients. Therefore, we assume that the source density is specified by two infinitely smooth functions:
$$
 f(x)= \left\{
  \begin{array}{ll}
    f^+(x), & x \in \Omega^+; \ \   f^+ \in C^\infty_0(\Omega^+),
\\[2pt]
  f^-(x), & x \in \Omega^-; \ \   f^- \in C^\infty_0(\Omega^-).
  \end{array}
\right.
$$

In accordance with the theory of boundary-value problems, it can be stated that for any fixed value of the parameter $\varepsilon,$
 there is a unique solution $u_\varepsilon$ to problem \eqref{original Problem}.

 The main objective of this paper is to construct an asymptotic expansion for the solution~$u_\varepsilon$ in the Sobolev space $H^1(\Omega_\varepsilon)$ as $\varepsilon \to 0$, along with proving the corresponding asymptotic estimates. Specifically, we derive the homogenized problem and higher-order transmission conditions, providing approximations for both the solution and its gradient in the  entire  partially perforated domain $\Omega_\varepsilon$, particularly near the interface $\mathcal{Z}$, with arbitrarily specified accuracy.

\begin{remark}
In all statements in the paper,  $\varepsilon= \frac{d}{N}$ is a discrete parameter $(N\in \Bbb N).$  Therefore, $\varepsilon \to 0$
means that $N \to +\infty.$
\end{remark}
\begin{remark}\label{Rem-2-2}
 The results of the article remain valid for the case of more general elliptic differential equations:
 $\mathrm{div}_x (A^\pm_\varepsilon \, \nabla_x u^\pm_\varepsilon) = f^\pm,$ where $A^\pm_\varepsilon =\{a^\pm_{i j}(\frac{x}{\varepsilon})\}$
 are  matrix with $\varepsilon$-periodic coefficients.  In this case, additional assumptions about the symmetry of the coefficients are necessary (see \cite{Melnyk-1995}).
\end{remark}

\section{Contraction of formal asymptotics}\label{Sect_3}

We start looking for the asymptotics in $\Omega^-.$ Since there is no any perturbation in this region, we seek an asymptotic expansion for the solution  in the form of a power series
\begin{equation}\label{as-V}
   V_\varepsilon^\infty:= \sum_{k=0}^{\infty}\varepsilon^k \, v^-_k(x) \ \ \text{in} \ \Omega^- .
\end{equation}
Then, formally substituting it into the Poisson equation and the Dirichlet boundary conditions, we get the following relations for the coefficients:
\begin{equation}\label{v_k-}
\Delta_x v^-_0(x) = f^-(x) \quad \text{and} \quad \Delta_x v^-_k(x) = 0 \ \  \text{in} \ \Omega^- , \qquad v^-_k(x) = 0 \ \ \text{on} \ \Gamma^- := \partial\Omega^- \setminus \mathcal{Z}.
\end{equation}

\subsection{Analysis in the perforated part}\label{subsect-3-1}

The methodology for constructing an asymptotic expansion in a strongly inhomogeneous periodic media is well established
 (see, e.g., \cite[Chapt. 4, \S 2]{Bach-Pan-1989}, \cite[Chapt. 7,]{Cio-Don-1999}, \cite[Chapt. 2, \S 4,]{Ole-Ios-Sha-1991}).  Consequently, we will briefly recall it and  focus more closely on examining how symmetry affects the asymptotic coefficients and the resulting implications.

In $\Omega_\varepsilon^+,$ we use a standard two-scale ansatz
\begin{equation}\label{as-U}
  U^\infty_\varepsilon := \sum_{k=0}^\infty\varepsilon^k\sum_{|\alpha|=k}{N_\alpha(\tfrac{x}{\varepsilon}) \, D^\alpha v_\varepsilon(x)},
\end{equation}
where $\alpha=(\alpha_1,\ldots,\alpha_k)$ is a multi-index,  $\alpha_i\in \{1, 2\},$
$|\alpha|=k$ is the number of the components of $\alpha$ and it is called its \textit{length},
$$
 D^\alpha v := \frac{\partial^k v}{\partial{x_{\alpha_1}}\ldots\partial{x_{\alpha_k}} },{} \quad  D^0 v = v,
$$
$N_\alpha(\xi)$ is  $1$-periodic in $\xi_1$ and $\xi_2,$ \ $\xi = \frac{x}{\varepsilon}=\big(\frac{x_1}{\varepsilon}, \, \frac{x_2}{\varepsilon}\big),$ \ $\xi =(\xi_1, \xi_2),$ {} \  $N_0\equiv 1.$

Using the chain rule and substituting  expansions \eqref{as-U} into the equation of problem \eqref{original Problem} and the Neumann conditions on the hole boundaries, and collecting terms with the same powers of $\varepsilon,$ we obtain
\begin{align*}
  \Delta_{x}U^\infty_\varepsilon =& \ \bigg(\varepsilon^{-1}\sum_{|\alpha|=1} \Delta_{\xi}N_\alpha(\xi) \, D^\alpha v_\varepsilon(x)
\\
&+ \sum_{k=2}^\infty\varepsilon^{k-2}\sum_{|\alpha|=k}\Big(\Delta_{\xi}N_\alpha(\xi) + 2 \,\partial_{\xi_{\alpha_1}}N_{\alpha_2 \alpha_3\ldots\alpha_k}(\xi)
+ \delta_{\alpha_1, \alpha_2} N_{\alpha_3\ldots\alpha_k}(\xi)\Big) \, D^\alpha v_\varepsilon(x)\bigg)\bigg|_{\xi =\frac{x}{\varepsilon}}
\\
\approx  & \ f^+(x), \quad x \in \Omega^+_\varepsilon,
\end{align*}
\begin{align*}
\partial_{\vec{\nu}_\varepsilon}U^\infty_\varepsilon  =& \  \bigg(\sum_{|\alpha|=1} \Big(\partial_{\vec{\nu}_\xi} N_\alpha(\xi)  + \nu_{\alpha}(\xi)\Big)D^\alpha v(x)
  \\
   & + \sum_{k=2}^\infty\varepsilon^{k-1}\sum_{|\alpha|=k}\Big(\partial_{\vec{\nu}_\xi}N_\alpha(\xi) + \nu_{\alpha_1}\!(\xi)\,
N_{\alpha_2\ldots\alpha_k}(\xi) \Big) D^\alpha v(x)\bigg)\bigg|_{\xi =\frac{x}{\varepsilon}}
\approx   0, \quad x \in \ \partial G_\varepsilon,
\end{align*}
where here and further  $\delta_{\alpha_1, \alpha_2}$ is the Kronecker delta.
\begin{remark}
  At the end of these equalities is  the symbol "$\approx$" that means we want
  that $U^\infty_\varepsilon$ to be an approximation to the solution.
    Then, we  have to determine the coefficients  to satisfy these equations.
\end{remark}

To satisfy this equations, we first neutralize the micro-variables $\xi$  by requiring that the coefficients $N_\alpha$
be solutions to the following problems, respectively:\\
for $|\alpha|=1,$   find  $N_1, N_2 \in H^1_{per}(Y) := \{N\in H^1(Y)\colon \ \ N \ \text{is 1-periodic in} \ \xi_1 \ \text{and} \ \xi_2\}:$
\begin{equation}\label{cell1}
\left\{
\begin{array}{rcll}
\Delta_{\xi}N_1(\xi) &=& 0 & \text{in} \ \ Y,
\\[2pt]
\partial_{\vec{\nu}_\xi} N_1(\xi) &=& -\nu_{1}(\xi) & \text{on} \ \ \partial G_0,
\\[2pt]
\langle N_1 \rangle_Y &=& 0,
\end{array}
\right.
\end{equation}
\begin{equation}\label{cell2}
\left\{
\begin{array}{rcll}
\Delta_{\xi}N_2(\xi) &=& 0 &  \text{in} \ \ Y,
\\[2pt]
\partial_{\vec{\nu}_\xi} N_2(\xi) &=& -\nu_{2}(\xi) &  \text{on} \ \ \partial G_0,
\\[2pt]
\langle N_2 \rangle_Y &=& 0;
\end{array}
\right.
\end{equation}
for $|\alpha|> 1, $ find  $N_\alpha \in H^1_{per}(Y)$:
\begin{equation}\label{cell3}
\left\{
\begin{array}{rcll}
\Delta_{\xi}N_\alpha(\xi)  &=& h_\alpha -  \delta_{\alpha_1, \alpha_2}\,  N_{\alpha_3\ldots\alpha_k}(\xi) - 2\, \partial_{\xi_{\alpha_1}}N_{\alpha_2 \alpha_3\ldots\alpha_k}(\xi) & \text{in} \ \ Y,
\\[3pt]
\partial_{\vec{\nu}_\xi}N_\alpha(\xi) &=& - \nu_{\alpha_1}\!(\xi)\,
N_{\alpha_2\ldots\alpha_k}(\xi)  &  \text{on} \ \ \partial G_0,
\\[3pt]
\langle N_\alpha \rangle_Y &=& 0,
\end{array}
\right.
\end{equation}
where $h_\alpha$ is a constant,
$$
\langle N \rangle_Y := \frac{1}{\upharpoonleft\!\! Y\!\!\upharpoonright } \int_Y N(\xi) \, d\xi, \qquad \upharpoonleft\!\! Y\!\!\upharpoonright \ := meas(Y).
$$

The solvability of this recurrent sequence of problems follows from the following lemma (for the proof see, e.g., \cite[Supplement, Th. 1]{Bach-Pan-1989}).

\begin{lemma}
  Let $F_0(\xi), \ F_1(\xi), F_2(\xi)$ be $1$-periodic in $\xi$ and smooth functions in $\overline{Y}.$ {}
Then there exists an unique smooth solution $N\in H^1_{per}(Y)$ to the problem
\begin{equation}\label{cell4}
\left\{
\begin{array}{rcll}
\Delta_{\xi}N(\xi)  &=& F_0(\xi) + \sum_{i=1}^{2}\partial_{\xi_i}F_i(\xi)  &  \text{in} \ \ Y,
\\[3pt]
\partial_{\vec{\nu}_\xi}N(\xi) &=&  \sum_{i=1}^{2} F_i(\xi)\, \nu_{i}(\xi) &  \text{on} \ \ \partial G_0,
\\[3pt]
\langle N \rangle_Y &=& 0,
\end{array}
\right.
\end{equation} {}
if and only if
$$
\langle F_0 \rangle_Y =0 .
$$
\end{lemma}
For problems \eqref{cell1} and \eqref{cell2},  $F_0 =0;$ for problem \eqref{cell3}, \
$F_0 = h_\alpha - \delta_{\alpha_1, \alpha_2}\, N_{\alpha_3\ldots\alpha_k} - \partial_{\xi_{\alpha_1}}N_{\alpha_2 \alpha_3\ldots\alpha_k}.$ Therefore, the constant
\begin{equation}\label{t1}
h_\alpha  =  \big\langle \delta_{\alpha_1, \alpha_2}\, N_{\alpha_3\ldots\alpha_k} + \partial_{\xi_{\alpha_1}}N_{\alpha_2 \alpha_3\ldots\alpha_k}\big\rangle_Y .
\end{equation}

Let us now examine how the symmetry of the periodicity cell affects the coefficients $N_\alpha.$  To do this, we introduce the reflection operator
$S_l$  with respect to the variable $\xi_l,$ $l\in \{1, 2\}$:
$$
S_l\xi = \big((-1)^{\delta_{l, 1}}\xi_1, \ (-1)^{\delta_{l, 2}}\xi_2 \big),
\ i.e.,
   \quad
  S_1\xi = (-\xi_1,  \xi_2 ),  {}
\quad S_2\xi = (\xi_1,  -\xi_2 ).
$$
It turns out that that if right-hand side of problem \eqref{cell4}  is either odd or even in some variable, then the solution inherits the same symmetry.
This effect for perforated domains was observed in \cite[Lemma 2.2]{Melnyk-1995}. For the convenience of the reader, I will present that statement in relation to problem \eqref{cell4}.
\begin{lemma}\label{Lemma-3-2}
  Let $N$ be a solution to problem \eqref{cell4}.  \ If, for some $l\in\{1, 2\},$ the function $F_0,$ $F_1,$ and $F_2$  satisfy
\begin{equation}\label{s1}
  F_i(S_l\xi) = (-1)^{\delta_{l, i}} F_i(\xi), \quad \xi \in \overline{Y}, \quad i\in\{0, 1, 2\},
\end{equation}
then the solution $N$ is even in $\xi_l,$ i.e.,
$
N(S_l\xi) = N(\xi), \  \xi \in \overline{Y}.
$

If, for some $l\in\{1, 2\},$
\begin{equation}\label{s2}
  F_i(S_l\xi) = (-1)^{\delta_{l, i}+1} F_i(\xi), \quad \xi \in \overline{Y}, \quad i\in\{0, 1, 2\},
\end{equation} {}
then the solution $N$ is odd  in $\xi_l,$ i.e.,
$
N(S_l\xi) = - N(\xi), \  \xi \in \overline{Y}.
$
\end{lemma}

\begin{remark}\label{remark-3-2}
In all symmetry relations we must write $\xi \in \bigcup\nolimits_{k,n \, \in \, \Bbb Z} \big( \overline{Y} +  (k, n)\big),$ but to shorten the writing
we indicate that $\xi \in  \overline{Y}.$

  For a 1-periodic function, the evenness (oddness) in some variable means the evenness (oddness) with respect to $\frac{1}{2},$ e.g.,
\ $F(\xi_1) = F(-\xi_1) = F(1 -\xi_1).$

 Due to the symmetry of the periodicity cell $Y$ the components of the normal satisfy \
 \begin{equation}\label{sym-normal}
  \nu_i(S_l\xi) = (-1)^{\delta_{l, i}} \nu_i(\xi), \quad  \xi \in \bigcup\nolimits_{k,n \, \in \, \Bbb Z} \big( \partial G_0 +  (k, n)\big),  \  \ i\in \{1, 2\}, \ \ l\in\{1, 2\}.
 \end{equation}
  \end{remark}

Let us apply Lemma~\ref{Lemma-3-2} to the recurrent procedure \eqref{cell1} - \eqref{cell3}. In virtue of   \eqref{sym-normal},
\begin{equation}\label{sym1}
  N_{\alpha}(S_l \xi) = (-1)^{\delta_{\alpha,l}} N_\alpha(\xi),\quad \xi \in \overline{Y}, \quad |\alpha|=1, \quad l\in \{1, 2\},
\end{equation}
i.e., $N_1$ is odd in $\xi_1$ and even in $\xi_2,$ and  $N_2$ is even in $\xi_1$ and odd  in $\xi_2.$

For $|\alpha|= 2$ we have problems
\begin{equation}\label{cell5}
\left\{
\begin{array}{rcl}
\Delta_{\xi}N_{\alpha_1 \alpha_2}(\xi)  &=& h_{\alpha_1 \alpha_2} -  \delta_{\alpha_1, \alpha_2}  - 2 \partial_{\xi_{\alpha_1}}N_{\alpha_2}(\xi) \ \  \text{in} \ \ Y,
\\[2pt]
\partial_{\vec{\nu}_\xi}N_{\alpha_1 \alpha_2}(\xi) &=& - \nu_{\alpha_1}\!(\xi)\,
N_{\alpha_2}(\xi)  \ \  \text{on} \ \ S_0,
\\[2pt]
\langle N_{\alpha_1 \alpha_2} \rangle_Y &=& 0.
\end{array}
\right.
\end{equation}
Let's clarify the symmetry of the right side.  On one side, based on $N_{\alpha_2}(S_l \xi) = (-1)^{\delta_{\alpha_2,l}} N_{\alpha_2}(\xi),$  we have
$$
\partial_{\xi_{\alpha_1}}\Big(N_{\alpha_2}(S_l\xi)\Big) = (-1)^{\delta_{\alpha_2,l}} \partial_{\xi_{\alpha_1}}N_{\alpha_2}(\xi).
$$
On the other, using the chain rule,  \ $\displaystyle \partial_{\xi_{\alpha_1}}\Big(N_{\alpha_2}(S_l\xi)\Big) = \partial_{\eta_{\alpha_1}}N_{\alpha_2}(\eta)\big|_{\eta= S_l\xi}\,  (-1)^{\delta_{\alpha_1,l}}.$
Thus,
$$
\partial_{\eta_{\alpha_1}}N_{\alpha_2}(\eta)\big|_{\eta= S_l\xi} =   (-1)^{\delta_{\alpha_1,l}\, + \, \delta_{\alpha_2,l}} \ \partial_{\xi_{\alpha_1}}N_{\alpha_2}(\xi).
$$
This means that if $\alpha_1 \neq \alpha_2,$  then  $h_{\alpha_1 \alpha_2} = \langle \partial_{\xi_{\alpha_1}}N_{\alpha_2} \rangle_Y = 0,$
and by Lemma~\ref{Lemma-3-2}  $N_{12}$ and $N_{21}$
are odd  in $\xi_1$ and $\xi_2.$
If $\alpha_1 = \alpha_2,$ then   $h_{\alpha_1 \alpha_1} = \langle 1+ \partial_{\xi_{\alpha_1}}N_{\alpha_1} \rangle_Y$  and  by Lemma~\ref{Lemma-3-2} $N_{\alpha_1 \alpha_1}$ is even in $\xi_1$ and $\xi_2.$

Summarising, we get
$$
N_{\alpha_1 \alpha_2}(S_l \xi) = (-1)^{\delta_{\alpha_1,l}\, + \, \delta_{\alpha_2,l}} \ N_{\alpha_1 \alpha_2}(\xi) ,\quad \xi \in \overline{Y}, \quad l\in \{1, 2\}.
$$

By using the method of mathematical induction, we can prove the lemma.
\begin{lemma}\label{Lemma-3-3}
  For any  $|\alpha| = k \ge 1,$ the  solution $N_\alpha$ to  problem \eqref{cell3}
satisfies
$$
   N_{\alpha}(S_l \xi) = (-1)^{\delta_{\alpha_1,l}\, + \, \delta_{\alpha_2,l} \, + \ldots +\, \delta_{\alpha_k,l}} \ N_{\alpha}(\xi) ,\quad \xi \in \overline{Y}, \quad l\in \{1, 2\}.
$$

In addition,
\begin{gather}
h_\alpha =0 \ \ \text{if} \ \ |\alpha| \ \ \text{is odd}, \label{h-1}
\\[2pt]
h_\alpha =0 \ \ \text{if} \ \   (-1)^{\delta_{\alpha_1,l}\, + \, \delta_{\alpha_2,l} \, + \ldots +\, \delta_{\alpha_k,l}} = -1 \ \
\text{at least at one} \ \ l\in \{1, 2\}. \label{h-2}
\end{gather}
 \end{lemma}

Thus,   the  coefficients $\{N_\alpha\}$ in series \eqref{as-U} are  determined, and
\begin{equation}\label{hom-eq1}
\Delta_x U^\infty_\varepsilon = \sum_{k=2}^\infty\varepsilon^{k-2}\sum_{|\alpha|=k} h_\alpha \, D^\alpha v_\varepsilon(x)
 \approx f^+(x), \quad x \in \Omega^+_\varepsilon.
 \end{equation} {}
The function $v_\varepsilon$  is sought in the form
\begin{equation}\label{as-v+}
  v_\varepsilon(x) := \sum_{n=0}^{\infty}\varepsilon^n v^+_n(x), \quad x \in \Omega^+.
\end{equation}

Substituting this series in \eqref{hom-eq1} and equating coefficients at the same power of $\varepsilon,$ we get
$$
\sum_{k=0}^\infty \varepsilon^{k}\Big(\sum_{|\alpha|=2} h_\alpha \, D^\alpha v^+_k(x) +  \sum_{n=0}^{k-1} \sum_{|\alpha|=k-n+2} h_\alpha \, D^\alpha v^+_n(x)\Big) \approx f^+(x), \quad x \in \Omega^+_\varepsilon.
$$
 To satisfy this relation, it is necessary to equate the terms with the same degree $\varepsilon$ from the left and right sides of this equality. As a result, we obtain a recurrent sequence of differential equations for the coefficients $\{v_k^+\}.$ At  $\varepsilon^0$ we get
$$
   \sum_{|\alpha|=2} h_\alpha \, D^\alpha v^+_0(x) = f^+(x)
\
\stackrel{\eqref{h-2}}{\Longleftrightarrow}
 \ {h_{11}} \, \partial^2_{x_1}v^+_0(x) + {h_{22}}\,  \partial^2_{x_2}v^+_0(x) = f^+(x)
 \ {\Longleftrightarrow} \ {\widehat{\mathcal{H}}}\, v^+_0 = f^+,
 $$
where \ \ ${\widehat{\mathcal{H}}} := {h_{11}} \, \partial^2_{x_1} \  + \  {h_{22}}\,  \partial^2_{x_2};$ \  at  $\varepsilon^1:$
$$
{\widehat{\mathcal{H}}}\,v^+_1(x) +  \sum_{|\alpha|= 3} h_\alpha \, D^\alpha v^+_0(x) = 0
\ \ \stackrel{\eqref{h-1}}{\Longleftrightarrow} \ \ {\widehat{\mathcal{H}}}\,v^+_1(x)  = 0;
$$
at $\varepsilon^k:$ \
$
{\widehat{\mathcal{H}}}\,v^+_k(x) = f^+_k(x),
$
 where
\begin{equation}\label{f_k}
   f^+_k(x) = -   \sum_{n=0}^{k-1} \sum_{|\alpha|=k-n+2} h_\alpha \, D^\alpha v^+_n(x).
\end{equation}

Similarly, as in, e.g.,  \cite[Chapt.4, \S 1]{Bach-Pan-1989} or in \cite[\S 6.3]{Cio-Don-1999}, we show that the differential operator $\widehat{\mathcal{H}}$ is elliptic, in particular
\begin{gather}\label{ell1}
  {h_{11}} = \big\langle 1 + \partial_{\xi_1}N_1\big\rangle_Y  =
\Big\langle\, \big|\partial_{\xi_1}(\xi_1 + N_1)\big|^2 + \big|\partial_{\xi_2} N_1)\big|^2 \Big\rangle_Y > 0,
  \\ \label{ell2}
  {h_{22}} = \big\langle 1 + \partial_{\xi_2}N_2\big\rangle_Y= \Big\langle\,  \big|\partial_{\xi_1}N_1\big|^2 + \big|\partial_{\xi_2}(\xi_2 +  N_1)\big|^2 \Big\rangle_Y > 0.
\end{gather}

Of course,  these  differential equations must be supplemented by the Dirichlet boundary conditions  on $\partial\Omega^+\setminus \mathcal{Z}.$ Thus, we get
\begin{equation}\label{v+0}
\left\{
\begin{array}{rcll}
{\widehat{\mathcal{H}}}\,v^+_0(x) &=& f^+(x), &\ \ x \in \Omega^+,
\\[4pt]
v^+_0(x) &=& 0, & \ \  x \in \Gamma^+ := \partial\Omega^+ \setminus {\mathcal{Z}},
\end{array}
\right.
\end{equation}
\begin{equation}\label{v+k}
\left\{
\begin{array}{rcll}
{\widehat{\mathcal{H}}}v^+_k(x) &=& f^+_k(x),  &  x \in \Omega^+,
\\[4pt]
v^+_k(x) &=& 0, & x \in \Gamma^+,
\end{array}
\right. \quad \text{for} \ \ k\in \Bbb N.
\end{equation}

Suppose we find $\{v_k^+\}.$ Let us find out additional properties of these functions. We start with $v_0^+.$ Since
$v^+_0(x_1,0) = 0$ for all $x_1 \in (0, d),$
\begin{equation}\label{v1}
  \partial^p_{x_1}v^+_0(x_1,0) = 0 \quad \text{on} \ \ \in (0, d) \quad \text{for all} \ \   p\in \Bbb N.
\end{equation}
Because of $f^+\in C^\infty_0(\Omega^+),$ there is a positive $\delta> 0$ such that
\begin{equation}\label{v2}
{h_{11}} \, \partial^2_{x_1 x_1}v^+_0(x_1,x_2) \,  + \, {h_{22}}\,  \partial^2_{x_2 x_2}v^+_0(x_1,x_2) =0 \quad \text{in} \ \ (0, d)\times(0, \delta).
\end{equation}
Passing to the limit in \eqref{v2} as $x_2 \to 0$ and considering \eqref{v1}, we get that
$ \partial^2_{x_2 x_2}v^+_0(x_1,0) =0$ on $(0, d),$
from which it follows that for all $p\in \Bbb N$
 \begin{equation}\label{v3}
   \partial^{p}_{x_1}\big(\partial^2_{x_2}v^+_0(x_1,0)\big), \quad x_1 \in (0, d).
\end{equation}
Differentiating the equation \eqref{v2} twice with respect to $x_2$, we obtain
\begin{equation}\label{v4}
{h_{11}} \, \partial^2_{x_1}\big(\partial^2_{x_2}v^+_0(x_1,x_2)\big)   +
{h_{22}}\,  \partial^4_{x_2}v^+_0(x_1,x_2) =0  \quad \text{in} \ \ (0, d)\times(0, \delta).
\end{equation}
Now passing to the limit in \eqref{v4} as $x_2 \to 0$ and taking \eqref{v3} into account, we derive
$$
\partial^4_{x_2}v^+_0(x_1,0) =0, \quad x_1 \in (0, d).
$$
Repeating these arguments, we conclude that $ \partial^{2p}_{x_2}v^+_0(x_1,0) = 0 $ on $(0, d)$ for all $ p\in \Bbb N.$

In a similar way we show that for \ $\forall \, p \in \Bbb N$
\begin{itemize}
  \item $\partial^p_{x_1}v^+_0(x_1,d) = 0$ \ and \   $\partial^{2p}_{x_2}v^+_0(x_1,d) = 0$ \ \ for \ \ $x_1\in(0, d);$
  \smallskip
  \item  $\partial^p_{x_2}v^+_0(d, x_2) = 0$ \ and \   $\partial^{2p}_{x_1}v^+_0(d, x_2) = 0$ \ \ for \ \ $x_2\in(0, d).$
\end{itemize}

Then, using the method of mathematical induction, we derive similar properties for derivatives of the remaining coefficients.
\begin{proposition}\label{Prop-3-1} The following relations hold:
 \begin{itemize}
    \item $\forall  p \in \Bbb N  \ \ \forall k \in \Bbb N_0\colon \quad \partial^p_{x_1}v^+_k(x_1,0) = 0, \ \ \partial^p_{x_1}v^+_k(x_1,d) = 0, \ \
\partial^p_{x_2}v^+_k(d, x_2) = 0;$
\smallskip
    \item $ \forall  p \in \Bbb N  \ \ \forall k \in \Bbb N_0\colon \quad  \partial^{2p}_{x_2}v^+_k(x_1,0) = 0, \ \ \partial^{2p}_{x_2}v^+_k(x_1,d) = 0, \ \
\partial^{2p}_{x_1}v^+_k(d, x_2) = 0. $
  \end{itemize}
\end{proposition}

These properties are the basis for our next statement.

\begin{proposition}\label{Prop-3-2}
The series \eqref{as-U} vanishes at $\Gamma^+ := \partial\Omega^+ \setminus {\mathcal{Z}}.$
\end{proposition}
\begin{proof}
 Let us show how to prove this for $x_2 =0.$ Consider its restriction
 $$
 U^\infty_\varepsilon\Big|_{x_2=0}  =
\sum_{k=0}^\infty\varepsilon^k \sum_{n=0}^\infty\varepsilon^n \sum_{|\alpha|=k} N_\alpha(\tfrac{x}{\varepsilon}) \,  D^\alpha v_n(x)\Big|_{x_2=0}
$$
If $k=0,$ then $v_n(x_1, 0) =0.$ Take any $k \in \Bbb N$ and show that $\displaystyle N_\alpha(\tfrac{x}{\varepsilon}) \,  D^\alpha v_n(x)\big|_{x_2=0} = 0.$

If the number of components of  $\alpha$ that equals  $2$ is odd,  then by Lemma~\ref{Lemma-3-3}
$$
N_{\alpha}(\xi_1, -\xi_2) = (-1)^{\delta_{\alpha_1,2} +  \delta_{\alpha_2, 2}  +\ldots+  \delta_{\alpha_k,2}} \, N_{\alpha}(\xi_1,\xi_2)  = -  N_{\alpha}(\xi_1,\xi_2)  \ \ {\Longrightarrow} \ \ N_{\alpha}\big|_{x_2=0} =0.
$$
If this number is even, then by Proposition~\ref{Prop-3-1} we have
$\displaystyle D^\alpha v^+_n(x)\big|_{x_2 =0} = 0. $
\end{proof}

Thus,  the series $U^\infty_\varepsilon$ formally satisfies the following relations:
$$
\Delta_{ x}U^\infty_\varepsilon =f^+  \quad  \text{in} \  \ \Omega^+_\varepsilon, \qquad
\partial_{\vec{\nu}_\varepsilon}U^\infty_\varepsilon = 0 \quad \text{on} \ \ \partial G_{\varepsilon},
\qquad
U^\infty_\varepsilon = 0 \quad \text{on} \ \ \Gamma^+ .
$$

\subsection{Inner-layer asymptotics}

To obtain transmission conditions for the coefficients $\{v_k^-\}$ and $\{v_k^+\}$ of series \eqref{as-V} and  \eqref{as-v+} respectively,
 we then run the inner-layer asymptotics in a vicinity of the interface interval $\mathcal{Z}$  between the perforated and non-perforated parts of the domain $\Omega_\varepsilon.$
The inner-layer ansatz is sought in the form of two series
\begin{equation}\label{as-B}
  B^\infty_\varepsilon  := \left\{
                          \begin{array}{ll}
\displaystyle
  \sum_{k=1}^\infty\varepsilon^k\sum_{|\alpha|=k} B^+_\alpha(\xi) \, D^\alpha v^+_\varepsilon(x)|_{x_1=0},
& \xi = \tfrac{x}{\varepsilon}, \ \ x \in \Omega^+_\varepsilon,
\\
\displaystyle
  \sum_{k=1}^\infty\varepsilon^k\sum_{|\alpha|=k} B^-_\alpha(\xi) \, D^\alpha v^+_\varepsilon(x)|_{x_1=0},
 & \xi = \tfrac{x}{\varepsilon}, \ \ x \in \Omega^-,
                          \end{array}
                        \right.
\end{equation}
where \ $v_\varepsilon^+(x) = \sum_{n=0}^{\infty}\varepsilon^n v^+_n(x),$ \ $B^\pm_0 \equiv 0,$ \
$ B_\alpha(\xi) := \left\{
                     \begin{array}{ll}
                       B^+_\alpha(\xi), & \xi \in \Upsilon^+, \\
                       B^-_\alpha(\xi), & \xi \in \Upsilon^-,
                     \end{array}
                   \right.
$
is $1$-periodic in $\xi_2,$ and it is also required that  $B^\pm_\alpha(\xi) \to 0$
as $\xi_1 \to \pm\infty.$
\begin{figure}[htbp]\label{F}
  \centering
  \includegraphics[width=10cm]{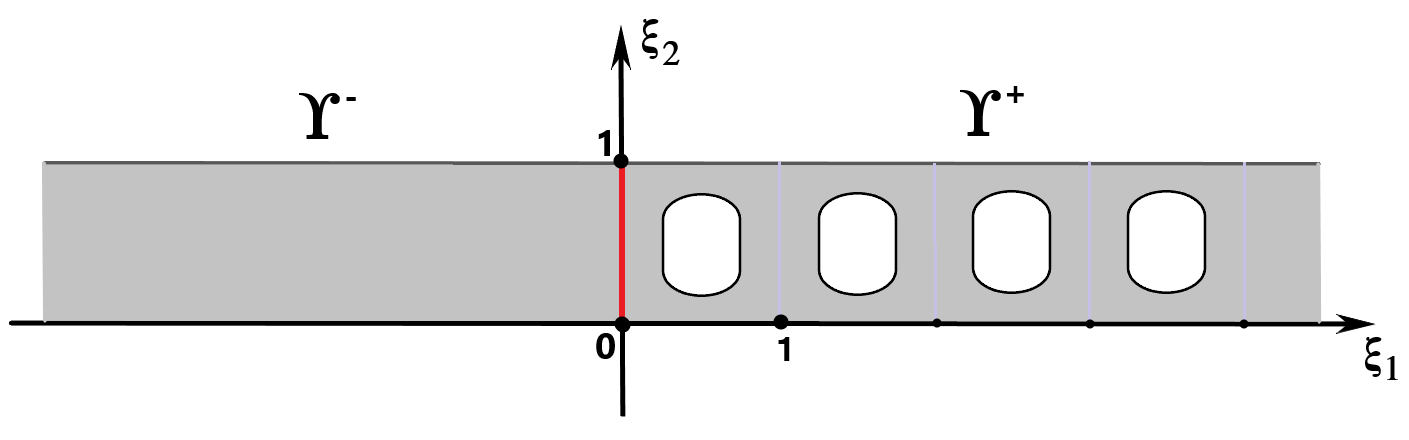}
  \caption{The partial perforated band-cell}\label{f3}
   \end{figure}

Here,  $\Upsilon^- := \{\xi\colon \ \xi_1 < 0, \ 0<\xi_2<1\}$  and $\Upsilon^{+,0} := \{\xi\colon \ 0 < \xi_1, \ 0<\xi_2<1\}$ are  infinite semi-strips  and
$$
\Upsilon^+ := \Upsilon^{+,0} \bigcap \Big(\bigcup_{k \in \Bbb N_0} \big( Y + (k, 0)\big)\Big)
$$
is $1$-periodic  perforated semi-strip  in the direction $\xi_1$ (see Fig.~\ref{f3}).

Since formally $\Delta_{ x}U^\infty_\varepsilon = f^+$ in  $\Omega^+_\varepsilon$  and  $\Delta_{ x}V^\infty_\varepsilon = f^-$  in   $\Omega^-,$ the Laplace operator of the series~$B^\infty_\varepsilon$ must be equal to zero.
 Using the same technique as before, we get the following equation
 \begin{align*}
  \Delta_{x}B^\infty_\varepsilon & = \bigg[\varepsilon^{-1}\sum_{|\alpha|=1}\Delta_{\xi}B^\pm_\alpha(\xi) \, D^\alpha v^+_\varepsilon(x)\big|_{x_1=0}
\\
&\ \ \ +
\sum_{k=2}^\infty\varepsilon^{k-2}\sum_{|\alpha|=k}
\Big(\Delta_{\xi}B^\pm_\alpha(\xi)  + 2 \delta_{\alpha_1, 2} \, \partial_{\xi_2}B^\pm_{\alpha_2\ldots\alpha_k}(\xi)
 +
\delta_{\alpha_1, 2} \, \delta_{\alpha_2, 2}\,  B^\pm_{\alpha_3\ldots\alpha_k}(\xi) \Big)  D^\alpha v^+_\varepsilon(x)\big|_{x_1=0}\bigg]\bigg|_{\xi =\frac{x}{\varepsilon}} \approx 0.
\end{align*}
Substituting $B^\infty_\varepsilon$ in the Neumann conditions, we obtain
\begin{align*}
\partial_{\vec{\nu}_\varepsilon}U^\infty_\varepsilon & = \bigg[ \sum_{|\alpha|=1} \partial_{\vec{\nu}_\xi} B^+_\alpha(\xi)\, D^\alpha v^+_\varepsilon(x)\big|_{x_1=0}
\\
& \ \ \ + \sum_{k=2}^\infty\varepsilon^{k-1}\sum_{|\alpha|=k}\Big(\partial_{\vec{\nu}_\xi}B^+_\alpha(\xi) + \delta_{\alpha_1, 2}\, \nu_{2}(\xi)\,
B^+_{\alpha_2\ldots\alpha_k}(\xi)\Big) D^\alpha v^+_\varepsilon(x)\big|_{x_1=0}\bigg]\bigg|_{\xi =\frac{x}{\varepsilon}} \approx  0 \ \quad \text{on} \ \ \partial G_\varepsilon.
\end{align*}
To formally satisfy these equations, we need to equate the sums of the terms depending on the variables $\xi$  to zero.
As a result, we obtain the following relations for coefficients of the inner-layer asymptotics:\\
for $|\alpha|=1$
\begin{equation}\label{B-1}
  \Delta_{\xi}B^\pm_\alpha(\xi)  =0 \ \ \text{in} \ \ \Upsilon^\pm, \qquad \partial_{\vec{\nu}_\xi} B^+_\alpha(\xi)  =0 \ \ \text{on} \ \ \mathfrak{S}^+
\quad (\text{the boundaries of the holes}),
\end{equation}
and for $|\alpha|\ge 2$
\begin{equation}\label{B-alpha}
\left\{
\begin{array}{l}
\displaystyle
\Delta_{\xi}B^\pm_\alpha(\xi)  + 2 \delta_{\alpha_1, 2} \, \partial_{\xi_2}B^\pm_{\alpha_2\ldots\alpha_k}(\xi) +
\delta_{\alpha_1, 2} \, \delta_{\alpha_2, 2}\,  B^\pm_{\alpha_3\ldots\alpha_k}(\xi) =0  \ \ \text{in} \ \ \Upsilon^\pm,
\\[5pt]
\displaystyle
\partial_{\vec{\nu}_\xi}B^+_\alpha(\xi) + \delta_{\alpha_1, 2}\, \nu_{2}(\xi)\,
B^+_{\alpha_2\ldots\alpha_k}(\xi) = 0 \ \ \text{on} \ \ \mathfrak{S}^+.
\end{array}
\right.
\end{equation}

To find  relations  for $B^-_\alpha$ and $B^+_\alpha$ at the interval  ${\gamma} := \{\xi\colon \ \xi_1 =0, \ \ \xi_2 \in (0, 1)\},$
we substitute series \eqref{as-V}, \eqref{as-U} and \eqref{as-B} in the transmission conditions
$$
u_\varepsilon(x)|_{x_1=-0} = u_\varepsilon(x)|_{x_1=+0} \ \ \text{and} \ \ \partial_{x_1}u_\varepsilon(x)|_{x_1=-0} = \partial_{x_1}u_\varepsilon(x)|_{x_1=+0} \ \ \text{at} \ \ {\mathcal{Z}}.
$$
The result is
\begin{equation*}
  \big(V_\varepsilon^\infty + B^{\infty,-}_\varepsilon\big)|_{x_1=-0} =  \big(U^{\infty}_\varepsilon + B^{\infty,+}_\varepsilon\big)|_{x_1=+0} \  \ {\Longleftrightarrow} \ \ V_\varepsilon^\infty|_{x_1=0}  =  U^{\infty}_\varepsilon|_{x_1=0}  +  \big(B^{\infty,+}_\varepsilon|_{x_1=+0} \, - \,  B^{\infty,-}_\varepsilon|_{x_1=-0}\big),
\end{equation*}
and
\begin{equation}\label{tran2}
\partial_{x_1}V_\varepsilon^\infty|_{x_1=0}  = \partial_{x_1}U^{\infty}_\varepsilon|_{x_1=0} \, + \,
\big(\partial_{x_1}B^{\infty,+}_\varepsilon|_{x_1=+0} \, - \, \partial_{x_1}B^{\infty,-}_\varepsilon|_{x_1=-0}\big).
\end{equation}
The first aforementioned relation, in its expanded form, is represented as follows
\begin{align*}
  \sum_{k=0}^{\infty}\varepsilon^k \, v^-_k(0,x_2)= & \  v^+_0(0,x_2) + \sum_{k=1}^\infty\varepsilon^k \Big(v^+_k(0,x_2) +  \sum_{n=0}^{k-1} \sum_{|\alpha|=k -n} N_\alpha(0,\xi_2) \,  D^\alpha v^+_n(x)\big|_{x_1=0}\Big)
    \\
   & + \sum_{k=1}^\infty\varepsilon^k \sum_{n=0}^{k-1}\sum_{|\alpha|=k-n} \big[B_\alpha(\xi) \big]_{\xi_1=0} \, D^\alpha v^+_n(x)\big|_{x_1=0},
\end{align*}
where \ $\displaystyle  \boldsymbol{\big[}B_\alpha(\xi) \boldsymbol{\big]}_{\xi_1=0} := B^+_\alpha(0,\xi_2) - B^-_\alpha(0,\xi_2)$
is the jump of the enclosed quantity.

Next, we should equate terms with the same degree of $\varepsilon$ from the left and right sides of this equality.  At $\varepsilon^0$  there is no problem, we get the equality $v^-_0(0,x_2)   = v^+_0(0,x_2),$ $x_2 \in \mathcal{Z}.$
For the following terms, the micro-variable $\xi_2$ is on the right side and needs to be neutralized.
So if it is possible to select
\begin{equation}\label{tran3+}
\big[B_\alpha(\xi) \big]_{\xi_1=0} = - N_\alpha(0,\xi_2) + {q_\alpha},
\end{equation}
where   ${q_\alpha}$ \ is a constant,   then
\begin{equation}\label{tran3}
 v^-_k(0,x_2) = v^+_k(0,x_2) +  \sum_{n=0}^{k-1} \sum_{|\alpha|=k -n} {q_\alpha} \,  D^\alpha v^+_n(x)\big|_{x_1=0}, \quad x_2 \in {\mathcal{Z}}.
\end{equation}

Relation \eqref{tran2}, in its expanded form, is represented as follows
\begin{align*}
 \sum_{k=0}^{\infty}\varepsilon^k \, \partial_{x_1}v^-_k(0,x_2) &  =
{\Big(\partial_{\xi_1}N_1(\xi) +1\Big)\big|_{\xi_1=0}\, \partial_{x_1}v^+_0(x)\big|_{x_1=0}}
 +  \sum_{k=1}^\infty\varepsilon^k  \bigg(\Big(\partial_{\xi_1}N_1(\xi) +1\Big)\big|_{\xi_1=0}\,
\partial_{x_1}v^+_k (x)\big|_{x_1=0}
\\
& +
 \sum_{n=0}^{k-1} \sum_{|\alpha|=k -n +1} {\Big(\partial_{\xi_1}N_\alpha(\xi) + \delta_{\alpha_1, 1} N_{\alpha_2 \ldots \alpha_{k-n+1}}(\xi)\Big)\big|_{\xi_1=0} } \, D^\alpha v^+_n(x)\big|_{x_1=0}\bigg)
\\
& + {\sum_{|\alpha|=1} \big[\partial_{\xi_1}B_\alpha \big]_{\xi_1=0} \, D^\alpha v^+_0(x)\big|_{x_1=0}}
 + \sum_{k=1}^\infty\varepsilon^{k}\bigg(\sum_{|\alpha|=1} \big[\partial_{\xi_1}B_\alpha \big]_{\xi_1=0} \, D^\alpha v^+_k(x)\big|_{x_1=0}
\\
&
+  \sum_{n=0}^{k-1}\sum_{|\alpha|=k+1-n} { \big[ \partial_{\xi_1}B_\alpha \big]_{\xi_1=0}} \, D^\alpha v^+_n(x)\big|_{x_1=0}\bigg).
\end{align*}
Equating terms with the same degree of $\varepsilon$ from the left and right sides of this equality, we conclude
\begin{itemize}
  \item at $\varepsilon^0 :$ \ \   $\big[\partial_{\xi_1}B_2 \big]_{\xi_1=0}=0;$   \ \ if  \ $\big[\partial_{\xi_1}B_1 \big]_{\xi_1=0} = - \Big(\partial_{\xi_1}N_1(\xi) +1\Big)\big|_{\xi_1=0} + {J_1},$ \ where ${J_1}$ is a constant,  then
\begin{equation}\label{tran-cond-2}
 \partial_{x_1}v^-_0(x)\big|_{x_1=0} = {J_1} \, \partial_{x_1}v^+_0(x)\big|_{x_1=0},
\end{equation}
\item at $\varepsilon^k :$ \ \  if  \ ${\big[\partial_{\xi_1}B_\alpha \big]_{\xi_1=0}} = - {\Big(\partial_{\xi_1}N_\alpha(\xi) + \delta_{\alpha_1, 1} N_{\alpha_2 \ldots \alpha_{k}}(\xi)\Big)\big|_{\xi_1=0}} + {J_\alpha},$ \ where $J_\alpha$ is a constant,  then
\begin{equation}\label{tran-cond-3}
 \partial_{x_1}v^-_k(x) \big|_{x_1=0} = {J_1} \, \partial_{x_1}v^+_k(x)\big|_{x_1=0} +
\sum_{n=0}^{k-1} \sum_{|\alpha|=k -n +1} {J_\alpha} \, D^\alpha v^+_n(x)\big|_{x_1=0}.
\end{equation}
\end{itemize}

As a result, we obtain the recurrent sequence of boundary-value problems to determine the coefficients $\{B_\alpha\} \colon$
\begin{equation}\label{Innner cell1}
\left\{
\begin{array}{l}
\Delta_{\xi}B^\pm_1(\xi) = 0 \ \  \text{in} \ \ \Upsilon^\pm, \qquad \partial_{\vec{\nu}_\xi} B^+_1(\xi) = 0\ \  \text{on} \ \ \mathfrak{S}^+,
\\[4pt]
\big[ B_1 \big]_{\xi_1=0}  =  {q_1},  \qquad
\big[\partial_{\xi_1}B_1 \big]_{\xi_1=0}
 =  - \big(\partial_{\xi_1}N_1(\xi) +1\big)\big|_{\xi_1=0} + {J_1} \ \ \text{at} \ \ {\gamma},
\\[4pt]
B^\pm_1(\xi) \to 0 \ \  \text{as} \ \ \xi_1\to \pm \infty, \qquad  B^\pm_1(\xi)\ \  \text{are} \ 1\text{-periodic in} \ \xi_2,
\end{array}
\right.
\end{equation}
\smallskip
\begin{equation}\label{Innner cell2}
\left\{
\begin{array}{l}
\Delta_{\xi}B^\pm_2(\xi) = 0 \ \  \text{in} \ \ \Upsilon^\pm, \qquad \partial_{\vec{\nu}_\xi} B^+_2(\xi) = 0\ \  \text{on} \ \ \mathfrak{S}^+,
\\[4pt]
\big[ B_2 \big]_{\xi_1=0} = - N_2(0,\xi_2) + {q_2},  \qquad
\big[\partial_{\xi_1}B_2 \big]_{\xi_1=0}  =  0 \ \ \text{at} \ \ {\gamma}
\\[4pt]
B^\pm_1(\xi) \to 0 \ \  \text{as} \ \ \xi_1\to \pm \infty, \qquad  B^\pm_1(\xi)\ \  \text{are} \ 1\text{-periodic in} \ \xi_2,
\end{array}
\right.
\end{equation}
for $|\alpha|\ge 2,$
\begin{equation}\label{Innner cell3}
\left\{
\begin{array}{l}
\Delta_{\xi}B^\pm_\alpha(\xi)  = -  2 \delta_{\alpha_1, 2} \, \partial_{\xi_2}B^\pm_{\alpha_2\ldots\alpha_k}(\xi) -
\delta_{\alpha_1, 2} \, \delta_{\alpha_2, 2}\,  B^\pm_{\alpha_3\ldots\alpha_k}(\xi)
\ \  \text{in} \ \ \Upsilon^\pm,
\\[4pt]
\partial_{\vec{\nu}_\xi}B^+_\alpha(\xi) = -  \delta_{\alpha_1, 2}\, \nu_{2}(\xi)\,
B^+_{\alpha_2\ldots\alpha_k}(\xi)  \ \  \text{on} \ \ \mathfrak{S}^+,
\\[4pt]
\big[B_\alpha \big]_{\xi_1=0}  = - N_\alpha(0,\xi_2) + {q_\alpha}  \ \ \text{at} \ \ {\gamma},
\\[4pt]
\big[\partial_{\xi_1}B_\alpha \big]_{\xi_1=0} = - \big(\partial_{\xi_1}N_\alpha(\xi) + \delta_{\alpha_1, 1} N_{\alpha_2 \ldots \alpha_{k}}(\xi)\big)\big|_{\xi_1=0} + {J_\alpha}  \ \ \text{at} \ \ {\gamma},
\\[4pt]
B^\pm_\alpha(\xi) \to 0 \ \  \text{as} \ \ \xi_1\to \pm \infty, \qquad  B^\pm_\alpha(\xi)\ \  \text{are} \ 1\text{-periodic in} \ \xi_2.
\end{array}
\right.
\end{equation}

To prove the solvability of this recurrent sequence of problems, consider the following model problem:
\ find \ $ B(\xi) = \left\{
                     \begin{array}{l}
                       B^+(\xi), \  \xi \in \Upsilon^+, \\
                       B^-(\xi), \ \xi \in \Upsilon^-,
                     \end{array}
                   \right.
$
 \  that solves the problem
\begin{equation}\label{Innner cell-model}
\left\{
\begin{array}{l}
\Delta_{\xi}B^\pm(\xi)  = F^\pm_0(\xi) +  \partial_{\xi_2}F^\pm_2(\xi)
\ \  \text{in} \ \ \Upsilon^\pm,
\\[4pt]
\partial_{\vec{\nu}_\xi}B^+(\xi) =  \nu_{2}(\xi)\,
F^+_2(\xi)  \ \  \text{on} \ \ \mathfrak{S}^+,
\\[4pt]
\big[B \big]_{\xi_1=0} = \Phi(\xi_2) + {q}  \quad  \text{and} \quad
\big[\partial_{\xi_1}B \big]_{\xi_1=0}
 = \Psi(\xi_2)  + {J}  \ \ \text{at} \ \ {\gamma},
\\[4pt]
B^\pm(\xi) \to 0 \ \  \text{as} \ \ \xi_1\to \pm \infty, \qquad  B^\pm(\xi)\ \  \text{are} \ 1\text{-periodic in} \ \xi_2
\end{array}
\right.
\end{equation}
\begin{theorem}\label{Th-3-1}
   Let the right-hand sides $F^\pm_0, \ F^\pm_2, \ \Phi, \ \Psi$ in problem \eqref{Innner cell-model} be smooth functions in their domains of definition and  1-periodic in $\xi_2.$
Let $e^{\delta_0\, |\xi_1|} \, F^\pm_0 \in L^2(\Upsilon^\pm), \ \ e^{\delta_0\, |\xi_1|} \, F^\pm_2 \in L^2(\Upsilon^\pm)$ for some $\delta_0 > 0$
 and
\begin{equation}
 { J} = - \int_{{\gamma}}\Psi(\xi_2) \, d\xi_2 - \int_{\Upsilon^\pm} F^\pm_0(\xi)\, d\xi .
\end{equation}

Then there exists a unique number ${q} \in \Bbb R$ and a unique  solution to problem \eqref{Innner cell-model}  with the following differentiable asymptotics
\begin{equation*}
  B(\xi) = \mathcal{O}\big(e^{-\delta\, |\xi_1|}\big) \ \ \text{as} \ \ |\xi_1| \to \infty \quad (\delta >0).
\end{equation*}

In addition,
\begin{itemize}
  \item if $F^\pm_0, \ \Phi, \ \Psi$ are odd in $\xi_2$ and $F^\pm_2$ are even in $\xi_2,$ then the solution $B$ is odd in $\xi_2$
and
$$
{q} =0 \quad \text{and} \quad {J}=0;
$$
\item if $F^\pm_0, \ \Phi, \ \Psi$ are even in $\xi_2$ and $F^\pm_2$ are odd in $\xi_2,$ then the solution $B$ is even  in $\xi_2.$
\end{itemize}
\end{theorem}

\begin{remark}
The first statement  of this theorem was presented in \cite{Jag-Ole-Sham-1993}, highlighting that the proof relies on constructing a sequence of solutions to boundary value problems in finite domains
$
\Upsilon_M := (\Upsilon^- \cup \Upsilon^+) \cap \{\xi\colon |\xi_1| <M\}
$
and then passing to the limit as $M\to +\infty.$ In this context, results of paper \cite{Ole-Ios} concerning the behavior of solutions of elliptic equations in cylindrical domains with periodic boundary conditions  were used. In the present paper, a complete proof will be provided using the new approach outlined in point 2.

The second statement presents a new result.
\end{remark}

\begin{proof} \textbf{1.} Let   $\Upsilon := \Upsilon^- \cup {\gamma} \cup \Upsilon^+$.
 First we  look  for a bounded solution $\widehat{B}$ to problem \eqref{Innner cell-model}  in the form:  \  $\widehat{B} = P_1 + P_2,$   where
$P_1$ is a bounded solution to the problem
\begin{equation}\label{Innner cell-model1}
\left\{
\begin{array}{l}
\Delta_{\xi}P^\pm_1(\xi)  = F^\pm_0(\xi) +  \partial_{\xi_2}F^\pm_2(\xi)
\ \  \text{in} \ \ \Upsilon^\pm,
\\[4pt]
\partial_{\vec{\nu}_\xi}P^+_1(\xi) =  \nu_{2}(\xi)\,
F^+_2(\xi)  \ \  \text{on} \ \ \mathfrak{S}^+,
\\[4pt]
\big[P_1 \big]_{\xi_1=0} = 0  \quad  \text{and} \quad
\big[\partial_{\xi_1}P_1 \big]_{\xi_1=0}
 = \Psi(\xi_2)  + {J}  \ \ \text{at} \ \ {\gamma},
\\[4pt]
P_1 \ \  \text{is bounded in} \ \ \Upsilon, \qquad  P^\pm_1(\xi)\ \  \text{are} \ 1\text{-periodic in} \ \xi_2,
\end{array}
\right.
\end{equation}
and $P_2$ is a bounded solution to the problem
\begin{equation}\label{Innner cell-model2}
\left\{
\begin{array}{l}
\Delta_{\xi}P^\pm_2(\xi)  = 0
\ \  \text{in} \ \ \Upsilon^\pm,
\\[4pt]
\partial_{\vec{\nu}_\xi}P^+_2(\xi) =  0 \ \  \text{on} \ \ \mathfrak{S}^+,
\\[4pt]
\big[ P_2 \big]_{\xi_1=0} = \Phi(\xi_2)  \quad  \text{and} \quad
\big[\partial_{\xi_1}P_2\big]_{\xi_1=0}
 = 0  \ \ \text{at} \ \ {\gamma},
\\[4pt]
P_2 \ \  \text{is bounded in} \ \ \Upsilon, \qquad  P^\pm_2(\xi)\ \  \text{are} \ 1\text{-periodic in} \ \xi_2 .
\end{array}
\right.
\end{equation}
  First we study the solvability of problem \eqref{Innner cell-model1}, and then problem \eqref{Innner cell-model2} is reduced to problem of type \eqref{Innner cell-model1} with zero jump for the solution.

 \smallskip

 {\bf 2.}
Let $C^{\infty}_{0, per \xi_2}(\overline{\Upsilon}) $ be a space of functions that are infinitely differentiable in $\overline{\Upsilon},$ $1$-periodic in $\xi_2$ and have finite supports  with respect to  $\xi_1$, i.e.,
$$
\forall \,\varphi\in C^{\infty}_{0,per \xi_2}(\overline{\Upsilon}) \quad \exists \,R>0 \quad \forall \, \xi\in\overline{\Upsilon} \quad |\xi_1| \geq R\, : \quad \varphi(\xi)=0.
$$
We now define a Hilbert  space  $\mathsf{H} := \overline{\left( C^{\infty}_{0,per \xi_2}(\overline{\Upsilon}), \ \| \cdot \|_\mathsf{H} \right)}$, where the norm $\| \cdot \|_\mathsf{H}$ is generated by the scalar product
$$
(\varphi, \psi)_\mathsf{H}
 =  \int_\Upsilon\nabla_\xi \varphi \cdot \nabla_\xi \psi \, d\xi + \int_\Upsilon \rho^2(\xi_1) \, \varphi \, \psi \, d\xi
$$
with the weight function  $ \rho \in C^{\infty}(\Bbb{R})$ such that $0\leq \rho \leq1$ and
$$
\rho (\xi_1) =
\left\{\begin{array}{ll}
1,            & \mbox{if} \quad                        |\xi_1| \le 1, \\
|\xi_1|^{-1}, & \mbox{if} \quad |\xi_1| \geq 2.
\end{array}\right.
$$
It should be noted that each function from the space $\mathsf{H}$  has the finite Dirichlet integral and constant functions belong to this space.

\begin{definition}
  A function $P_1 \in \mathsf{H}$  is called a weak solution to  problem \eqref{Innner cell-model1} if the identity
\begin{equation}\label{integr}
    \int_{\Upsilon} \nabla_\xi P_1 \cdot \nabla_\xi \varphi \, d\xi =  \int_{\Upsilon^\pm} F_2^\pm  \, \partial_{\xi_2}\varphi \, d\xi - \int_{\Upsilon^\pm} F_0^\pm  \, \varphi \, d\xi - \int_{{\gamma}}\big(\Psi(\xi_2) + {J}\big) \, \varphi(0,\xi_2) \, d\xi_2
\end{equation}
holds for all $\varphi \in \mathsf{H}$.
\end{definition}

Hereinafter the symbol $\int_{\Upsilon^\pm} F^\pm \,d\xi$ means the sum $\int_{\Upsilon^+} F^+\,d\xi + \int_{\Upsilon^-} F^- \,d\xi.$
Now we prove the statement.
\begin{proposition}\label{Prop-3-3}
 Let $\rho^{-1} F^\pm_0 \in L^2(\Upsilon), \ F^\pm_2 \in L^2(\Upsilon), \ \Psi \in L^2({\gamma}).$

There exists a  weak solution $P_1 \in \mathsf{H}$  to  problem \eqref{Innner cell-model1} if and only if
\begin{equation}\label{solv-cond}
{ J} = - \int_{{\gamma}}\Psi(\xi_2) \, d\xi_2 - \int_{\Upsilon^\pm} F^\pm_0(\xi)\, d\xi .
\end{equation}
This solution is defined up to an additive constant.
\end{proposition}
\begin{proof}
The proof of necessity follows immediately from  identity \eqref{integr}, taking the test function equal to 1. Sufficiency is demonstrated as follows.
First, considering the inclusions $\rho^{-1} F^\pm_0 \in L^2(\Upsilon),$ $F^\pm_2 \in L^2(\Upsilon),$ $\Psi \in L^2({\gamma}),$ it is easy to
prove that the right-hand side of identity \eqref{integr} is a linear and bounded functional  over the space $\mathsf{H}.$

The left side of identity \eqref{integr} can be rewritten as follows
$$
\int_{\Upsilon} \nabla P_1 \cdot \nabla \varphi \, d\xi = \langle P_1, \varphi \rangle - \int_{\Upsilon_2} P_1 \,  \varphi \, d\xi,
$$
where $\Upsilon_k := \Upsilon \cap \{\xi\colon |\xi_1| <k\}$  and
\begin{equation}\label{new-scalar-prod}
  \langle \psi, \varphi \rangle := \int_\Upsilon \nabla_\xi \psi \cdot \nabla_\xi \varphi \, d\xi + \int_{\Upsilon_2} \psi \, \varphi \, d\xi.
\end{equation}
Then the new scalar product \eqref{new-scalar-prod} generates an  equivalent norm in $\mathsf{H}.$ It is obvious that
$\langle \psi,\psi\rangle \le c_1 \|\psi\|^2_{\mathsf{H}},$ $\psi \in\mathsf{H}.$
To prove the inverse inequality with another constant, it  suffices to show that
\begin{equation}\label{add-ineq}
  \int_{\Upsilon^+} \rho^2(\xi_1) \, \psi^2 \, d\xi \le C_1 \Big(\int_{\Upsilon^+_2} \psi^2 \, d\xi + \int_{\Upsilon^+} |\nabla_\xi \psi|^2 \, d\xi\Big) \quad \text{for any} \ \ \psi\in \mathsf{H},
\end{equation}
where $\Upsilon^+_k := \Upsilon^+ \cap \{\xi\colon 0 < \xi_1 <k\}.$ To do this, we use  Hardy's inequality
\begin{equation}\label{Hardy}
\int^{+\infty}_{0}(1+\xi_1)^{-2} \, \phi^2(\xi_1)\,d\xi_1 \le
4\int^{+\infty}_{0}|\partial_{\xi_1}\phi|^2\,d\xi_1,\qquad \phi\in H^1((0,+\infty))\, ,\ \ \phi(0)=0 ,
\end{equation}
a linear extension operator $\mathcal{P}\colon H^1(Y) \mapsto H^1(\square)$ (see, e.g., \cite{Cio-Paulin1079}) such that $\mathcal{P} u = u$ if $u = const,$ and
\begin{equation}\label{extention}
  \|\mathcal{P}u\|_{H^1(\square)} \le c_2 \|u\|_{H^1(Y)} \quad \text{and} \quad
  \|\nabla_\xi \mathcal{P}u\|_{L^2(\square)} \le c_3 \|\nabla_\xi u\|_{L^2(Y)},
\end{equation}
where the constants $c_2$ and $c_3$ are independent of $u \in H^1(Y),$ and the cut-off function $\chi \in C^{\infty}(\Bbb R),$ $ 0\le\chi\le 1,$
$$
\chi(\xi_1)= \left\{
               \begin{array}{ll}
                 0, & \xi_1 \le 1,\\
                 1, & \xi_1 \ge 2.
               \end{array}
             \right.
$$

Then for any function $\psi \in \mathsf{H}$ we have
\begin{equation}\label{in1}
  \int_{\Upsilon^+} \rho^2(\xi_1) \, \psi^2 \, d\xi \le  \int_{\Upsilon^+_2} \psi^2 \, d\xi + \int_{\Upsilon^+} \chi^2(\xi_1)\, \rho^2(\xi_1) \, \psi^2 \, d\xi.
\end{equation}
Next we estimate the second integral in the right-hand side of \eqref{in1}:
\begin{align}
  \int_{\Upsilon^+}  \rho^2(\xi_1) \, \chi^2(\xi_1)\, \psi^2 \, d\xi = &  \sum_{k=0}^{+\infty}\int_{\Upsilon^+_{k, k+1}}  \rho^2(\xi_1) \, \chi^2(\xi_1)\, \psi^2 \, d\xi
\le  \sum_{k=0}^{+\infty}\int_{\Upsilon^{+,0}_{k, k+1}}  \rho^2(\xi_1) \, \chi^2(\xi_1)\, (\mathcal{P}\psi)^2 \, d\xi \notag
\\
  = &  \int_{0}^{1}\Big(\int_0^{+\infty} \rho^2(\xi_1) \, \chi^2(\xi_1)\, (\mathcal{P}\psi)^2 \, d\xi_1\Big) d\xi_2 \notag
\\
\stackrel{\eqref{Hardy}}{\le} &\ C_2  \int_{0}^{1}\Big(\int_0^{+\infty} \big|\partial_{\xi_1}\big(\chi(\xi_1)\, \mathcal{P}\psi\big) \big|^2 \, d\xi_1\Big) d\xi_2
\notag
\\
\le &\ 2\,  C_2 \int_{\Upsilon^{+,0}} \big(|\chi'(\xi_1)\, \mathcal{P}\psi|^2 + \chi^2(\xi_1)\, |\partial_{\xi_1}\mathcal{P}\psi|^2\big) \, d\xi
\notag
\\
\le & \ C_3 \Big( \int_{\Upsilon^{+,0}_{1, 2}} |\mathcal{P}\psi|^2 \, d\xi + \sum_{k=0}^{+\infty}\int_{\Upsilon^{+,0}_{k, k+1}} |\partial_{\xi_1}\mathcal{P}\psi|^2\, d\xi \Big)
\notag
\\
\stackrel{\eqref{extention}}{\le} &\  C_4 \Big(\int_{\Upsilon^{+}_{1, 2}}(\psi^2 + |\nabla_\xi \psi|^2) \, d\xi + \sum_{k=0}^{+\infty}\int_{\Upsilon^{+}_{k, k+1}} |\nabla_{\xi}\psi|^2\, d\xi \Big)
\notag
\\
\le &\ C_5 \Big( \int_{\Upsilon^{+}_{1, 2}} \psi^2 \, d\xi + \int_{\Upsilon^+} |\nabla_{\xi}\psi|^2\, d\xi \Big) .\label{in2}
\end{align}
From \eqref{in1} and \eqref{in2} it follows inequality  \eqref{add-ineq}. Here $\Upsilon^{+,0}_{k, k+1} := \Upsilon^{+,0} \cap \{\xi\colon \ k<\xi_1 < k+1\},$
$\Upsilon^{+}_{k, k+1} := \Upsilon^{+} \cap \{\xi\colon \ k<\xi_1 < k+1\}.$ It is obvious that $\Upsilon^{+,0}_{k, k+1}$ and $\Upsilon^{+}_{k, k+1}$ are the translation of the unit square $\square$ and the periodicity cell $Y$, respectively, onto the vector $(k,0),$ $k\in \Bbb N.$
Therefore, we use the same notation for the extension operator from $H^1(\Upsilon^{+}_{k, k+1})$ into $H^1(\Upsilon^{+,0}_{k, k+1})$   as in  \eqref{extention}; clearly the constants are the same.

As the embedding  $ \mathsf{H} \subset  L^{2}(\Upsilon_2) $ is compact, there exists a self-adjoint positive compact operator
$\mathbf{A} : \mathsf{H} \mapsto \mathsf{H}$ such that
$$
 \langle \mathbf{A}\psi, \varphi \rangle =    \int_{\Upsilon_2} \psi \, \varphi \, d\xi \, , \quad \{ \psi, \varphi \}
                                        \in \mathsf{H} .
$$
 Thus, using the Riesz representation theorem, we can rewrite  identity \eqref{integr} as the operator equation
$$
P_1 - \mathbf{A} P_1 = \mathcal{F},
$$
and apply  the Fredholm theory to it. It is obvious that every solution of the corresponding homogeneous problem for problem \eqref{Innner cell-model1}
in the space $\mathsf{H}$  is a constant (its Dirichlet integral is  trivial). Therefore,  equality \eqref{solv-cond}  is the solvability condition for  problem \eqref{Innner cell-model1}. \end{proof}

From Proposition~\ref{Prop-3-1}  and from  a theorem on the behavior of solutions to elliptic equations in cylindrical domains with periodic boundary conditions  \cite{Ole-Ios} (see also \cite[Chapt. I, \S 8]{Ole-Ios-Sha-1991}) it follows the statement.

\begin{proposition}\label{Prop-3-4}
 Let $e^{\delta_0\, |\xi_1|} \, F^\pm_0 \in L^2(\Upsilon^\pm), \ \ e^{\delta_0\, |\xi_1|} \, F^\pm_2 \in L^2(\Upsilon^\pm)$ for some $\delta_0 > 0,$
and  equality \eqref{solv-cond} be satisfied.
\ Then there exists a unique solution $P_1 \in \mathsf{H}$  to  problem \eqref{Innner cell-model1} with the following differentiable asymptotics:
\begin{equation}\label{inner_asympt_general}
P_1(\xi)=\left\{
\begin{array}{rl}
  C_1 +  {\mathcal O}(e^{\delta_1 \xi_1}) & \mbox{as} \ \ \xi_1\to -\infty \quad (\delta_1 >0),
\\[2mm]
     {\mathcal O}(e^{- \delta_2 \xi_1}) & \mbox{as} \ \ \xi_1\to+\infty \quad (\delta_2 >0),
\end{array}
\right.
\end{equation}
where $C_1$ is a constant.
\end{proposition}

\smallskip

 {\bf 3.} By substitution
$$
\displaystyle T(\xi) = \left\{
                                                      \begin{array}{ll}
                                                        P^-_2(\xi), & \xi \in \Upsilon^-, \\[2pt]
                                                        P^+_2(\xi) - \Phi(\xi_2)\, \eta(\xi_1), & \xi \in \Upsilon^+,
                                                      \end{array}
                                                    \right.
$$
where $\displaystyle  \eta(\xi_1) = \left\{
                                                      \begin{array}{ll}
                                                      1, & \xi_1  \in [0, \tau/2], \\[2pt]
                                                        0, & \xi_1 \ge \tau,
                                                      \end{array}
                                                    \right.
$ and $\tau$ is the distant from $\gamma$ till $G_0$,
problem \eqref{Innner cell-model2} is reduced  to
\begin{equation}\label{Innner cell-model4}
\left\{
\begin{array}{l}
\Delta_{\xi}T^-(\xi)  = 0
\ \  \text{in} \ \ \Upsilon^-,
\qquad \Delta_{\xi}T^+(\xi)  = - \Phi''(\xi_2) \, \eta(\xi_1) - \Phi(\xi_2) \, \eta''(\xi_1)
\ \  \text{in} \ \ \Upsilon^+,
\\[4pt]
\partial_{\vec{\nu}_\xi}T^+(\xi) =  0 \ \  \text{on} \ \ \mathfrak{S}^+,
\\[4pt]
\big[T \big]_{\xi_1=0} = 0   \quad  \text{and} \quad
\big[\partial_{\xi_1}T\big]_{\xi_1=0}
 = 0  \ \ \text{at} \ \ \ {\gamma}
\\[4pt]
T \ \  \text{is bounded in} \ \ \Upsilon, \qquad  T \ \  \text{are} \ 1\text{-periodic in} \ \xi_2 .
\end{array}
\right.
\end{equation}

It is easy verify that the solvability condition for this problem is satisfied and from Propositions \ref{Prop-3-3} and \ref{Prop-3-4}  it follows that there is a unique solution with the differentiable asymptotics
\begin{equation*}
T(\xi)=\left\{
\begin{array}{rl}
C_2 +  {\mathcal O}(e^{\delta_3 \xi_1}) & \mbox{as} \ \ \xi_1\to -\infty \quad (\delta_3 >0),
\\[2mm]
     {\mathcal O}(e^{- \delta_4 \xi_1}) & \mbox{as} \ \ \xi_1\to+\infty, \quad (\delta_4 >0),
\end{array}
\right.
\end{equation*}
where $C_2$ is a constant.

Then  $\displaystyle P_2(\xi) = \left\{
                                                      \begin{array}{ll}
                                                        T^-(\xi), & \xi \in \Upsilon^-, \\[2pt]
                                                        T^+(\xi) + \Phi(\xi_2)\, \eta(\xi_1), & \xi \in \Upsilon^+,
                                                      \end{array}
                                                    \right.
$
and $P_2$ has the same asymptotics as $T.$

Thus, the bounded solution $\widehat{B} = P_1 + P_2$  to problem \eqref{Innner cell-model} has the asymptotics
\begin{equation*}
\widehat{B}(\xi) = \left\{
\begin{array}{rl}
  C_1 + C_2 +  {\mathcal O}(e^{\delta_5 \xi_1}) & \mbox{as} \ \ \xi_1\to -\infty \quad (\delta_5 >0),
\\[2mm]
     {\mathcal O}(e^{- \delta_6 \xi_1}) & \mbox{as} \ \ \xi_1\to+\infty, \quad (\delta_6 >0),
\end{array}
\right.
\end{equation*} {}
and the desired solution to the model problem \eqref{Innner cell-model} is determined by the formula
$$
B(\xi) = \left\{
                                                      \begin{array}{ll}
                                                        \widehat{B}^-(\xi) - q, & \xi \in \Upsilon^-, \\[2pt]
                                                        \widehat{B}^+(\xi), & \xi \in \Upsilon^+,
                                                      \end{array}
                                                    \right.
\quad \text{where} \ \ q := C_1  + C_2.
$$

\smallskip

 {\bf 4.} Now let us prove the second statement of the theorem. Suppose that $F^\pm_0, \ \Phi, \ \Psi$ are odd in $\xi_2$ and $F^\pm_2$ are even in $\xi_2.$ This means that the right-hand sides $F^\pm_0 +  \partial_{\xi_2}F^\pm_2$ in the differential equations of problem \eqref{Innner cell-model}
are odd functions in $\xi_2.$ Since the second component $\nu_2(\xi)$ of the normal is odd in $\xi_2$ with respect to $\frac{1}{2},$ the right-hand side
$\nu_2 \, F^+_2$ in the Neumann condition on $\mathfrak{S}^+$  has the same symmetry.

Taking into account these facts and Remark~\ref{remark-3-2}, it is easy to verify that
the function
$$
 B(\xi_1, 1 - \xi_2) = \left\{
                     \begin{array}{l}
                       B^+(\xi_1, 1 - \xi_2), \  \xi \in \Upsilon^+, \\
                       B^-(\xi_1, 1 - \xi_2), \ \xi \in \Upsilon^-,
                     \end{array}
                   \right.
$$
 solves the problem
\begin{equation}\label{Innner cell-model+}
\left\{
\begin{array}{l}
\Delta_{\xi}\big(B^\pm(\xi_1, 1 - \xi_2)\big)  = - F^\pm_0(\xi) -  \partial_{\xi_2}F^\pm_2(\xi)
\ \  \text{in} \ \ \Upsilon^\pm,
\\[4pt]
\partial_{\vec{\nu}_\xi}\big(B^+(\xi_1, 1 - \xi_2)\big) = -  \nu_{2}(\xi)\,
F^+_2(\xi)  \ \  \text{on} \ \ \mathfrak{S}^+,
\\[4pt]
\big[B(\xi_1, 1 - \xi_2)\big]_{\xi_1=0} = - \Phi(\xi_2) + {q}  \quad  \text{and} \quad
\big[\partial_{\xi_1}\big(B(\xi_1, 1 - \xi_2)\big)\big]_{\xi_1=0}
 = - \Psi(\xi_2)  + {J}  \ \ \text{at} \ \ {\gamma},
\\[4pt]
B^\pm(\xi_1, 1 - \xi_2) \to 0 \ \  \text{as} \ \ \xi_1\to \pm \infty, \qquad  B^\pm(\xi_1, 1 - \xi_2)\ \  \text{are} \ 1\text{-periodic in} \ \xi_2 .
\end{array}
\right.
\end{equation}

 Thus,  the sum $ B(\xi_1, \xi_2) +  B(\xi_1, 1 - \xi_2),$ $\xi \in \Upsilon,$ satisfies relations
\begin{equation*}
\left\{
\begin{array}{l}
\Delta_{\xi}\big(B^\pm(\xi_1, \xi_2) + B^\pm(\xi_1, 1 - \xi_2)\big)  = 0
\ \  \text{in} \ \ \Upsilon^\pm,
\\[4pt]
\partial_{\vec{\nu}_\xi}\big(B^+(\xi_1, \xi_2) + B^+(\xi_1, 1 - \xi_2)\big) = 0 \ \  \text{on} \ \ \mathfrak{S}^+,
\\[4pt]
\big[B(\xi_1, \xi_2) + B(\xi_1, 1 - \xi_2)\big]_{\xi_1=0} =  2 {q}  \quad  \text{and} \quad
\big[\partial_{\xi_1}\big(B(\xi_1, \xi_2) + B(\xi_1, 1 - \xi_2)\big)\big]_{\xi_1=0}
 = 2 {J}  \ \ \text{at} \ \ {\gamma},
\\[4pt]
B^\pm(\xi_1, \xi_2) + B^\pm(\xi_1, 1 - \xi_2) \to 0 \ \  \text{as} \ \ \xi_1\to \pm \infty, \quad  B^\pm(\xi_1, \xi_2) + B^\pm(\xi_1, 1 - \xi_2)\ \  \text{are} \ 1\text{-periodic in} \ \xi_2 .
\end{array}
\right.
\end{equation*}
By virtue of the first statement of Theorem~\ref{Th-3-1} we have $J =0,$ $q=0$ and
$B(\xi_1, \xi_2) + B(\xi_1, 1 - \xi_2) = 0$ for $\xi \in \overline{\Upsilon},$ or
$$
B(\xi_1,  - \xi_2) = - B(\xi_1, \xi_2), \quad \xi \in \bigcup_{k\in \Bbb Z} \big(\overline{\Upsilon}  + (0,k)\big).
$$
Similarly, it is proved that
 if $F^\pm_0, \ \Phi, \ \Psi$ are even in $\xi_2$ and $F^\pm_2$ are odd in $\xi_2,$ then  $B$ is even  in $\xi_2.$
\end{proof}

Now we return to solving the cell problems \eqref{Innner cell1} - \eqref{Innner cell3}.
By Theorem~\ref{Th-3-1} 1,  if
\begin{equation}\label{form-J_1}
  {J_1} = \int_{{\gamma}} \big(\partial_{\xi_1}N_1(\xi) +1\big)\big|_{\xi_1=0}\, d\xi_2,
\end{equation}
then there exists a unique number
${q_1} \in \Bbb R$ and a unique  solution $B_1$ to problem~\eqref{Innner cell1}, which exponentially decreases  as $|\xi_1| \to \infty$ and  is even in $\xi_2.$
\begin{proposition}\label{Prop-3-5}
  The number
\begin{equation}\label{form-J_1+}
{J_1} = \, \upharpoonleft\!\! Y\!\!\upharpoonright  {h_{11}},
\end{equation}
where $h_{11}$ is determined by formula \eqref{ell1}.
\end{proposition}
\begin{proof} The differential equation  of problem \eqref{cell1}  can be rewritten as follows
$$
 \partial_{\xi_1}\big(1 + \partial_{\xi_1}N_1\big) + \partial^2_{\xi_2}N_1 =0 \quad \text{in} \ \ Y.
$$
We multiply it by $(-\xi_1 +1)$ and integrate by parts in $Y.$ As a result, we get
$$
0  = \int_{\partial Y} \Big( \big(1 + \partial_{\xi_1}N_1\big) \nu_1(\xi) + \partial_{\xi_2}N_1 \nu_2(\xi)\Big) (-\xi_1 + 1)\, d \sigma_\xi
- \int_{Y} \big(1 + \partial_{\xi_1}N_1\big) \partial_{\xi_1}(-\xi_1 +1) \, d\xi
$$
Taking into account $1$-periodicity of $N_1$ and the Neumann condition on $\partial G_0,$ the previous equality is equivalent to the equality
 $$
0  = - \int_{\gamma} \big(1 + \partial_{\xi_1}N_1\big)\big|_{\xi_1=0}\, d\xi_2
+ \int_{Y} \big(1 + \partial_{\xi_1}N_1\big) \, d\xi,
$$
from which, considering \eqref{ell1},  follows relation \eqref{form-J_1+}.
\end{proof}

Since $N_2$ is odd in $\xi_2$ (see \eqref{sym1}),  by Theorem~\ref{Th-3-1} we have that ${q_2} =0$ and the solution $B_2$  to problem~\eqref{cell2}  decreases exponentially as $|\xi_1| \to \infty$ and
 is  odd in~$\xi_2.$
Thus, for $|\alpha|=1$ the following symmetry relations hold:
$$
B_{\alpha}(\xi_1, -\xi_2) = (-1)^{\delta_{\alpha,2}} \, B_{\alpha}(\xi_1, \xi_2), \ \ \xi \in \overline{\Upsilon}.
$$

Taking into account these properties of $B_1$ and $B_2$ and applying Theorem~\ref{Th-3-1} to problems
\begin{equation*}
\left\{
\begin{array}{l}
\Delta_{\xi}B^\pm_{\alpha_1 \alpha_2}(\xi)  = -  2 \delta_{\alpha_1, 2} \, \partial_{\xi_2}B^\pm_{\alpha_2}(\xi)  =0
\ \  \text{in} \ \ \Upsilon^\pm,
\\[4pt]
\partial_{\vec{\nu}_\xi}B^+_{\alpha_1 \alpha_2}(\xi) = - \delta_{\alpha_1, 2}\, \nu_{2}(\xi)\,
B^+_{\alpha_2}(\xi)  \ \  \text{on} \ \ \mathfrak{S}^+,
\\[4pt]
\big[B_{\alpha_1 \alpha_2} \big]_{\xi_1=0}
 = - N_{\alpha_1 \alpha_2}(0,\xi_2) + {q_{\alpha_1 \alpha_2}}  \ \ \text{at} \ \ {\gamma},
\\[4pt]
\big[\partial_{\xi_1} B_{\alpha_1 \alpha_2} \big]_{\xi_1=0}
 = - \big(\partial_{\xi_1}N_{\alpha_1 \alpha_2}(\xi) + \delta_{\alpha_1, 1} N_{\alpha_2}(\xi)\big)\big|_{\xi_1=0} + {J_{\alpha_1 \alpha_2}}  \ \ \text{at} \ \ {\gamma},
\\[4pt]
B^\pm_{\alpha_1 \alpha_2}(\xi) \to 0 \ \  \text{as} \ \ \xi_2\to \pm \infty, \qquad  B^\pm_{\alpha_1 \alpha_2}(\xi)\ \  \text{are} \ 1\text{-periodic in} \ \xi_2,
\end{array}
\right.
\end{equation*}
with  $|\alpha| = 2,$ we conclude  that  there are exponentially decreasing solutions to these problems  and
\begin{itemize}
  \item $\displaystyle B_{\alpha_1 \alpha_2}(\xi_1, -\xi_2) = (-1)^{\delta_{\alpha_1,2}\, + \, \delta_{\alpha_2,2}} \, B_{\alpha_1 \alpha_2}(\xi_1, \xi_2),$ \ \
$ \xi \in \overline{\Upsilon};$

\smallskip

\item if $\delta_{\alpha_1,2} +  \delta_{\alpha_2,2}$ is an odd number, then ${q_{\alpha_1 \alpha_2}} = {J_{\alpha_1 \alpha_2}} =0,$ {} i.e., ${q_{12}}={q_{21}}= {J_{12}}= {J_{21}}=0;$

\smallskip

\item $\displaystyle {J_{11}} = \int_{{\gamma}} \big(\partial_{\xi_1} N_{11} + N_1\big)|_{\xi_1=0}\, d\xi_2$ and
$\displaystyle
{J_{22}} = \int_{{\gamma}} \big(\partial_{\xi_1}N_{22}\big)|_{\xi_1=0}\, d\xi_2 + \int_{\Upsilon^\pm} \partial_{\xi_2} B^\pm_2\, d\xi .
$
 \end{itemize}

Then, by means of the method of mathematical induction, we can prove the following lemma.
\begin{lemma}
  The recurrent sequence of problems \eqref{Innner cell1} - \eqref{Innner cell3}
is uniquely solvable.   Furthermore, for each multi-index $\alpha, \ |\alpha|=k,$
\begin{itemize}
  \item $B_\alpha(\xi) = \mathcal{O}\big(e^{-\delta\, |\xi_1|}\big) \ \ \text{as} \ \ |\xi_1| \to \infty \quad (\delta >0);$

\smallskip

  \item $\displaystyle B_{\alpha}(\xi_1, -\xi_2) = (-1)^{\delta_{\alpha_1,2} + \ldots + \delta_{\alpha_k,2}} \, B_{\alpha}(\xi_1, \xi_2),$
\ \ $ \xi \in \overline{\Upsilon};$

\smallskip

  \item if $\delta_{\alpha_1,2} + \ldots + \delta_{\alpha_k,2}$ is an odd number, then ${q_{\alpha}} = {J_{\alpha}} =0.$
\end{itemize}
\end{lemma}

In the same way as the proof of Proposition~\ref{Prop-3-2}, we establish the following.

\begin{proposition}\label{Prop-3-6}
$$
B^\infty_\varepsilon(x) = 0 \quad \text{if} \ \ x \in \partial\Omega \cap \{x\colon \ x_2 =0\} \quad \text{or} \quad
x \in \partial\Omega \cap \{x\colon \ x_2 =d\}.
$$
\end{proposition}

Thus, we can determine all coefficients of the inner-layer asymptotics \eqref{as-B}. In addition, the transmission conditions \eqref{tran3}, \eqref{tran-cond-2} and \eqref{tran-cond-3} for the coefficients $\{v_k^\pm\}_{k\in \Bbb N_0}$ are found, which together with the equations \eqref{v_k-}, \eqref{v+0} and \eqref{v+k} form boundary value problems for them.

\subsection{Homogenized problem and problems for $\{v^\pm_k\}$}\label{par-3-3}

The function $ v_0(x) = \left\{
                     \begin{array}{ll}
                        v^+_0(x), & x \in \Omega^+, \\
                        v^-_0(x), & x  \in \Omega^-,
                     \end{array}
                   \right.
$
must be a solution to the problem
\begin{equation}\label{Homo-problem}
\left\{
\begin{array}{lcl}
\Delta_x v^-_0(x) = f^-(x) \ \ \text{in} \ \ \Omega^- , & &  \upharpoonleft\!\! Y\!\!\upharpoonright {h_{11}} \, \partial^2_{x_1}v^+_0(x) \  + \  \upharpoonleft\!\! Y\!\!\upharpoonright {h_{22}}\,  \partial^2_{x_2}v^+_0(x) = \ \upharpoonleft\!\! Y\!\!\upharpoonright f^+(x) \ \ \text{in} \ \ \Omega^+ ,
\\[2pt]
v^-_0(x)=0 \ \ \text{on} \ \ \Gamma^-,& &
v^+_0(x) = 0  \ \ \text{on} \ \ \Gamma^+,
\\[2pt]
v^-_0(0,x_2)   = v^+_0(0,x_2) \ \ \text{on} \ \ {\mathcal{Z}}, & &
\partial_{x_1}v^-_0(x) =\  \upharpoonleft\!\! Y\!\!\upharpoonright  {h_{11}} \, \partial_{x_1}v^+_0(x) \ \ \text{on} \ \ {\mathcal{Z}}.
\end{array}
\right.
\end{equation}
Problem \eqref{Homo-problem} is called a \emph{homogenized problem} for problem \eqref{original Problem}.

The coefficients $\{v_k^\pm\}_{k\in \Bbb N}$ must be solutions to the following problems
\begin{equation}\label{rec1}
\left\{
\begin{array}{l}
\Delta_x v^-_1(x) = 0 \ \ \text{in} \ \ \Omega^-, \hspace{1cm}   \upharpoonleft\!\! Y\!\!\upharpoonright {h_{11}} \, \partial^2_{x_1}v^+_1(x) \  + \  \upharpoonleft\!\! Y\!\!\upharpoonright {h_{22}}\,  \partial^2_{x_2}v^+_1(x) = 0 \ \ \text{in} \ \ \Omega^+ ,
\\[4pt]
v^-_1(x)=0 \ \ \text{on} \ \ \Gamma^-, \qquad
v^+_1(x) = 0  \ \ \text{on} \ \ \Gamma^+,
\\[4pt]
v^-_1(0,x_2)   = v^+_1(0,x_2) + {q_1}\, \partial_{x_1}v_0^+(x)|_{x_1=0}\ \ \text{on} \ \ {\mathcal{Z}},
\\[4pt]
\partial_{x_1}v^-_1(x) = \ \upharpoonleft\!\! Y\!\!\upharpoonright  {h_{11}} \, \partial_{x_1}v^+_1(x) + {J_{11}} \, \partial^2_{x_1}v^+_0(x)|_{x_1=0} +
 {J_{22}} \, \partial^2_{x_2}v^+_0(x)|_{x_1=0}  \ \ \text{on} \ \ {\mathcal{Z}},
\end{array}
\right.
\end{equation}{}
and
\begin{equation}\label{rec2}
\left\{
\begin{array}{l}
\Delta_x v^-_k(x) = 0 \ \ \text{in} \ \ \Omega^-, \hspace{0.9cm}   \upharpoonleft\!\! Y\!\!\upharpoonright {h_{11}} \, \partial^2_{x_1}v^+_k   +   \upharpoonleft\!\! Y\!\!\upharpoonright {h_{22}}\,  \partial^2_{x_2}v^+_k = \ \upharpoonleft\!\! Y\!\!\upharpoonright\,  f^+_k(x) \ \ \text{in} \ \ \Omega^+,
\\[4pt]
v^-_1(x)=0 \ \ \text{on} \ \ \Gamma^-, \qquad
v^+_1(x) = 0  \ \ \text{on} \ \ \Gamma^+,
\\[4pt]
\displaystyle
v^-_k(0,x_2)   = v^+_k(0,x_2) + \sum_{n=0}^{k-1} \sum_{|\alpha|=k -n} {q_\alpha} \,  D^\alpha v^+_n(x)\big|_{x_1=0} \ \ \text{on} \ \ {\mathcal{Z}},
\\[4pt]
\displaystyle
\partial_{x_1}v^-_1(x) = \ \upharpoonleft\!\! Y\!\!\upharpoonright  {h_{11}} \, \partial_{x_1}v^+_1(x) +
\sum_{n=0}^{k-1} \sum_{|\alpha|=k -n +1} {J_\alpha} \, D^\alpha v^+_n(x)\big|_{x_1=0},
\end{array}
\right.
\end{equation}
where the function $f^+_k$ is defined in \eqref{f_k}.

It is easy to show that there exists a unique weak solution to the homogenized problem.
Moreover, taking into account the zero Dirichlet boundary conditions and the fact that $f_k^\pm\in C^\infty_0(\Omega^\pm)$, and using the established results on the smoothness of solutions to boundary-value problems for elliptic equations
with discontinuous coefficients (see \cite{Oleinik-1961,Shish-1961,Scheftel-1965}), we conclude that
$
v_0^\pm \in C^\infty\big(\overline{\Omega^\pm} \big).
$
Note that to apply these results to the corner points, we use the standard odd extension for the solution through the corresponding side
and properties of its derivatives (see Proposition~\ref{Prop-3-1}).
For example, to prove $C^\infty$-regularity at   the point $(0,0)$, we consider the relations
\begin{equation*}
\left\{
\begin{array}{lcl}
\Delta_x \tilde{v}^-_0(x) = \tilde{f}^-(x) \ \ \text{in} \ [- \varrho_0, 0]\times [-d, d], & &  \upharpoonleft\!\! Y\!\!\upharpoonright {h_{11}} \, \partial^2_{x_1}\tilde{v}^+_0(x) \  + \  \upharpoonleft\!\! Y\!\!\upharpoonright {h_{22}}\,  \partial^2_{x_2}\tilde{v}^+_0(x) = \ \upharpoonleft\!\! Y\!\!\upharpoonright \tilde{f}^+(x) \ \ \text{in} \ \ \tilde{\Omega}^+ ,
\\[2pt]
\tilde{v}^-_0(0,x_2)   = \tilde{v}^+_0(0,x_2) \ \ \text{on} \ \ \tilde{\mathcal{Z}}, & &
\partial_{x_1}\tilde{v}^-_0(x) =\  \upharpoonleft\!\! Y\!\!\upharpoonright  {h_{11}} \, \partial_{x_1}\tilde{v}^+_0(x) \ \ \text{on} \ \ \tilde{\mathcal{Z}},
\end{array}
\right.
\end{equation*}
where $\tilde{v}^\pm_0$ and $\tilde{f}^\pm$ are the odd extensions of  ${v}^\pm_0$ and ${f}^\pm,$ respectively,
\ the constant $\varrho_0$ is defined in Sect.~\ref{Sec-2}, \
$\tilde{\mathcal{Z}} := \{x\colon x_1=0, \ \  -d< x<d\},$ \ $\tilde{\Omega}^+ := \{ x\colon \ 0 < x_1 < d, \ \ -d < x_2 < d\}.$
Then one should apply Theorem 10 from \cite{Scheftel-1965}.

In the same way, based on results of \cite{Scheftel-1965}, the existence, uniqueness and smoothness of solutions to  problems  \eqref{rec1} or \eqref{rec2}  with non-zero transmission conditions  are substantiated.  Thus,
$
v_k^\pm \in C^\infty\big(\overline{\Omega^\pm} \big)
$
for any $k\in \Bbb N.$

\section{Justification and the main results}\label{Sect-4}

By using  series  \eqref{as-V}, \eqref{as-U}, \eqref{as-v+}  and \eqref{as-B}, we  determine the series
\begin{equation}\label{whole-series}
 \mathcal{A}^\infty_\varepsilon
:=
\left\{\begin{array}{ll}
                                  \displaystyle v^+_0(x) +  \sum_{k=1}^\infty\varepsilon^k \bigg(v^+_k(x) +  \sum_{n=0}^{k-1} \sum_{|\alpha|=k -n} N_\alpha(\tfrac{x}{\varepsilon}) \,  D^\alpha v^+_n(x) &
\\[10pt]
\hspace{2.5cm} \displaystyle
 + \chi_0(x_1) \sum_{n=0}^{k-1} \sum_{|\alpha|=k -n} B^+_\alpha(\tfrac{x}{\varepsilon}) \, D^\alpha v^+_n(x)\big|_{x_1=0} \bigg),& x \in \Omega^+_\varepsilon,
\\[10pt]
    \displaystyle
v^-_0(x) +  \sum_{k=1}^\infty\varepsilon^k \bigg(v^-_k(x) +  \chi_0(x_1) \sum_{n=0}^{k-1} \sum_{|\alpha|=k -n} B^-_\alpha(\tfrac{x}{\varepsilon}) \, D^\alpha v^+_n(x)\big|_{x_1=0}\bigg), & x \in \Omega^-,
                                     \end{array}
                                   \right.
\end{equation}
where $\chi_0$ is a  smooth cut-off function such that
\begin{equation}\label{cut-off}
  \chi_0(x_1) = \left\{
                \begin{array}{ll}
                  1, & \hbox{if} \ \  |x_1| \le \frac{\varrho_0}{2},
\\
                  0, & \hbox{if} \ \ |x_1| \ge \varrho_0.
                \end{array}
              \right.
\end{equation}
\begin{theorem}\label{Theorem-main}
 Series \eqref{whole-series} is an asymptotic expansion for the solution $u_\varepsilon$ to problem \eqref{original Problem} in the Sobolev space
$H^1(\Omega_\varepsilon)$
and for any $m\in\Bbb N,$ there exist positive constants $C_m$  and $\varepsilon_0$ such that for all $\varepsilon \in (0, \varepsilon_0)$
\begin{equation}\label{main-est}
  \left\|u_\varepsilon -  \mathcal{A}^{(m)}_\varepsilon \right\|_{H^1(\Omega_\varepsilon)} \le C_m\, \varepsilon^{m},
\end{equation}
where $\mathcal{A}^{(m)}_\varepsilon$ is the partial sum of $\mathcal{A}^\infty_\varepsilon.$
\end{theorem}
\begin{remark}
  Hereinafter, all constants in inequalities are independent of the parameter $\varepsilon.$
\end{remark}
\begin{proof}
 {\bf 1.}
We first check that  for any $m\in \Bbb N$ the partial sum $\mathcal{A}^{(m)}_\varepsilon \in H^1(\Omega_\varepsilon):$
\begin{align*}
\big[ \mathcal{A}^{(m)}_\varepsilon \big]_{x_1=0}  & = \mathcal{A}^{(m)}_\varepsilon\big|_{x=+0} - \mathcal{A}^{(m)}_\varepsilon\big|_{x=-0}
\\
   & =  \sum_{k=1}^{m}\bigg(
v^+_k(0,x_2) +  \sum_{n=0}^{k-1} \sum_{|\alpha|=k -n} \Big(N_\alpha|_{x_1=0}  +
\big[B_\alpha \big]_{x_1=0}\Big) \,
 D^\alpha v^+_n(x)\big|_{x_1=0} - v^-_k(0,x_2)\bigg)
\\
  &\stackrel{\eqref{tran3+}}{=} \sum_{k=1}^{m}\bigg(
v^+_k(0,x_2) +  \sum_{n=0}^{k-1} \sum_{|\alpha|=k -n} {q_\alpha} \,
 D^\alpha v^+_n(x)\big|_{x_1=0} - v^-_k(0,x_2)\bigg) \stackrel{\eqref{tran3}}{=} 0 \quad \text{on} \ \ {\mathcal{Z}} .
\end{align*}

Using \eqref{tran-cond-2} and  \eqref{tran-cond-3}, we find
$$
\big[\partial_{x_1}\mathcal{A}^{(m)}_\varepsilon \big]_{x_1=0}
  = \varepsilon^{m} \, \Psi^{(m)}_\varepsilon(x_2),
$$
where
\begin{equation}\label{res1}
  \displaystyle \sup_{x_2 \in {\mathcal{Z}}} |\Psi^{(m)}_\varepsilon(x_2)| \le \check{C}_m.
\end{equation}

Due to Propositions~\ref{Prop-3-2} and \ref{Prop-3-6} and the boundary conditions in  \eqref{Homo-problem} - \eqref{rec2} we have
$
\mathcal{A}^{(m)}_\varepsilon|_{\partial\Omega} =0.
$

{\bf 2.} In this step of the proof we find the remainders that  the partial sum $\mathcal{A}^{(2m)}_\varepsilon$ leaves in the differential equation and the Neumann boundary conditions of problem \eqref{original Problem}.

  In $\Omega^+_\varepsilon$ this partial sum can be rewritten as follows
$$
\mathcal{A}^{(2m)}_\varepsilon(x) = U^{(m-1)}_\varepsilon(x) + \chi_0(x_1) \, B^{(m-1)}_\varepsilon(x)  + \varepsilon^m v^+_m(x) + \varepsilon^{m+1} \, R^{(m+1)}_\varepsilon(x),
$$
where
$$
U^{(m-1)}_\varepsilon(x) = \sum_{k=0}^{m+1}\varepsilon^k\sum_{|\alpha|=k}{N_\alpha(\tfrac{x}{\varepsilon}) \, D^\alpha v^{(m-1)}_\varepsilon(x)},
\quad  \ \  v^{(m-1)}_\varepsilon(x) = \sum_{n=0}^{m-1}\varepsilon^n v^+_n(x),
 $$
$$
B^{(m-1)}_\varepsilon(x)  = \sum_{k=1}^{m+1}\varepsilon^k\sum_{|\alpha|=k} B^\pm_\alpha(\tfrac{x}{\varepsilon}) \, D^\alpha v^{(m-1)}_\varepsilon(x)|_{x_1=0},
$$
and $\varepsilon^{m+1} \, R^{(m+1)}_\varepsilon$ is the sum of the remaining members.

Using the calculations performed in \S~\ref{subsect-3-1}, we find
\begin{align*}
  \Delta_{x}U^{(m-1)}_\varepsilon  - f^+ =& \ \sum_{k=0}^{m-1}\varepsilon^{k}\sum_{|\alpha|=k+2}\Big(\Delta_{\xi}N_\alpha + 2 \partial_{\xi_{\alpha_1}}N_{\alpha_2 \alpha_3\ldots\alpha_k}
+ \delta_{\alpha_1, \alpha_2} N_{\alpha_3\ldots\alpha_k}\Big)\Big|_{\xi =\frac{x}{\varepsilon}} \, D^\alpha  v^{(m-1)}_\varepsilon - f^+
\\
&\ \ +  \varepsilon^{m}\sum_{|\alpha|=m}\Big(2 \partial_{\xi_{\alpha_1}}N_{\alpha_2 \alpha_3\ldots\alpha_{m}}
+ \delta_{\alpha_1, \alpha_2} N_{\alpha_3\ldots\alpha_{m}}\Big)\Big|_{\xi =\frac{x}{\varepsilon}} \, D^\alpha  v^{(m-1)}_\varepsilon(x)
\\
& \ \ +
\varepsilon^{m+1}\sum_{|\alpha|=m+1} \delta_{\alpha_1, \alpha_2} N_{\alpha_3\ldots\alpha_{m+1}}\big|_{\xi =\frac{x}{\varepsilon}} \, D^\alpha  v^{(m-1)}_\varepsilon(x)
\\
 \stackrel{\eqref{cell3}}{=}  & \ \sum_{k=0}^{m-1}\varepsilon^{k}\sum_{n=0}^{m-1}\varepsilon^n \sum_{|\alpha|=k+2}h_\alpha  \, D^\alpha  v^+_n(x) - f^+(x)  + \mathcal{O}(\varepsilon^{m})
\\
 \stackrel{\eqref{v+k}}{=}
 & \ \varepsilon^{m} \, R^{(m)}_{1, \, \varepsilon}(x), \quad x \in \Omega^+_\varepsilon,
\end{align*}
 where
\begin{equation}\label{res2}
\sup_{x\in \Omega^+_\varepsilon} |R^{(m)}_{1, \, \varepsilon}(x)| \le \hat{C}_m .
\end{equation}
It is also easy to check that
\begin{align*}
\partial_{\vec{\nu}_\varepsilon}U^{(m-1)}_\varepsilon(x) & =  \sum_{k=1}^{m}\varepsilon^{k}\sum_{|\alpha|=k+1}\Big(\partial_{\vec{\nu}_\xi}N_\alpha + \nu_{\alpha_1}\,
N_{\alpha_2\ldots\alpha_{k+1}}\Big)\big|_{\xi =\frac{x}{\varepsilon}}\, D^\alpha v_\varepsilon^{(m-1)}(x)
\\
& \quad \ + \
\varepsilon^{m+1}\sum_{|\alpha|=m+2}\nu_{\alpha_1}\,
N_{\alpha_2\ldots\alpha_{m+2}}\big|_{\xi =\frac{x}{\varepsilon}}\, D^\alpha v_\varepsilon^{(m-1)}(x)
\\
& \stackrel{\eqref{cell3}}{=} \varepsilon^{m+1}\sum_{|\alpha|=m+1}\nu_{\alpha_1}\,
N_{\alpha_2\ldots\alpha_{m+3}}\big|_{\xi =\frac{x}{\varepsilon}}\, D^\alpha v_\varepsilon^{(m-1)}(x)
 =: \varepsilon^{m+1} R^{(m)}_{2, \, \varepsilon}(x)
 \ \quad x \in  \partial G_\varepsilon,
\end{align*}
where
\begin{equation}\label{res3}
\sup_{x\in \partial G_\varepsilon} |R^{(m)}_{2, \, \varepsilon}(x)| \le \breve{C}_m.
\end{equation}

In the same way, but now using \eqref{Innner cell1} - \eqref{Innner cell3} and taking into account the exponential decrease of the coefficients  $\{B_\alpha\},$ we find
$$
\Delta_{x}\big(\chi_0 \,B^{(m-1)}_\varepsilon\big) = \varepsilon^{m} \, \mathcal{R}^{(m)}_{1, \, \varepsilon} \ \ \text{in} \ \ \Omega^- \cup \Omega^+_\varepsilon,
\qquad
\partial_{\vec{\nu}_\varepsilon}\big(\chi_0 \,B^{(m-1)}_\varepsilon\big) = \varepsilon^{m+1} \mathcal{R}^{(m)}_{2, \, \varepsilon} \ \ \text{on} \ \  \partial G_\varepsilon,
$$
where
\begin{equation}\label{res4}
 \sup_{x\in \Omega^- \cup \, \Omega^+_\varepsilon} |\mathcal{R}^{(m)}_{1, \, \varepsilon}(x)| +  \sup_{x\in \partial G_\varepsilon} |\mathcal{R}^{(m)}_{2, \, \varepsilon}(x)| \le \tilde{C}_m .
\end{equation}

Thus, for any $m\in\Bbb N, $ the difference \ $ \mathcal{A}^{(2m)}_\varepsilon -  u_\varepsilon$  satisfies relations
\begin{equation}\label{difference}
\left\{
\begin{array}{ll}
\Delta_x\big(\mathcal{A}^{(2m)}_\varepsilon -  u_\varepsilon\big) =  \varepsilon^{m}  \mathcal{R}^{-(m)}_{1, \, \varepsilon} \ \ \text{in} \ \Omega^- , &
\Delta_x\big(\mathcal{A}^{(2m)}_\varepsilon -  u_\varepsilon\big) = \varepsilon^{m} \, R^{(m)}_{1, \, \varepsilon} +
\varepsilon^{m} \, \mathcal{R}^{+(m)}_{1, \, \varepsilon}
  \ \text{in} \ \Omega^+_\varepsilon
\\[2pt]
&  \partial_{\vec{\nu}_\varepsilon}\big(\mathcal{A}^{(2m)}_\varepsilon -  u_\varepsilon\big) = \varepsilon^{m+1} R^{(m)}_{2, \, \varepsilon} \ \ \text{on} \ \ \partial G_\varepsilon,
\\[4pt]
\big[\mathcal{A}^{(2m)}_\varepsilon - u_\varepsilon \big]_{x_1=0}
 = 0 \ \ \text{on} \ \ {\mathcal{Z}},
&  \big[\partial_{x_1}\big(\mathcal{A}^{(2m)}_\varepsilon -  u_\varepsilon\big)\big]_{x_1=0} = \varepsilon^{2m} \Psi^{(2m)}_\varepsilon(x_2) \ \ \text{on} \ \ {\mathcal{Z}},
\\[4pt]
\mathcal{A}^{(2m)}_\varepsilon -  u_\varepsilon=0  \ \ \text{on} \ \ \partial\Omega.
\end{array}
\right.
\end{equation}
Multiplying the differential equations in \eqref{difference} by $ \mathcal{A}^{(2m)}_\varepsilon -  u_\varepsilon,$ then integrating by parts and using \eqref{res1} - \eqref{res4}, we deduce the inequality
$$
\left\| \mathcal{A}^{(2m)}_\varepsilon - u_\varepsilon \right\|_{H^1(\Omega_\varepsilon)} \le \ddot{C}_m\, \varepsilon^{m},
$$
whence follows  inequality~\eqref{main-est}.
\end{proof}

In the context of applied problems, it is evident that there is no necessity to construct a complete asymptotic expansion for the solution. Instead, it suffices to employ an approximation of the solution that meets the requisite level of accuracy.  Therefore, weaker assumptions on the smoothness of the given functions can be considered. For example, if $m=1$ in \eqref{main-est}, then we should construct the approximation $\mathcal{A}^{(2)}_\varepsilon;$
thus, it suffices that $f^- \in C^1_0(\Omega^-)$ and $f^+ \in C^5_0(\Omega^+).$ The following statement holds.

\begin{corollary}
Let   $f^- \in C^1_0(\Omega^-)$ and $f^+ \in C^5_0(\Omega^+).$ Then
\begin{equation}\label{two-term est}
  \left\|u_\varepsilon -  \mathcal{A}^{(1)}_\varepsilon \right\|_{H^1(\Omega_\varepsilon)} \le C_1\, \varepsilon,
\end{equation}
where
\begin{equation*}
 \mathcal{A}^{(1)}_\varepsilon =
\left\{\begin{array}{ll}
                                  \displaystyle v^+_0(x) + \varepsilon \, \Big(v^+_1(x) +  N_1(\tfrac{x}{\varepsilon}) \, \partial_{x_1}v^+_0(x) +
N_2(\tfrac{x}{\varepsilon}) \, \partial_{x_2}v^+_0(x) &
\\
\hspace{2.7cm} + \ \chi_0(x_1) \, B^+_1(\tfrac{x}{\varepsilon}) \, \partial_{x_1}v^+_0(x)\big|_{x_1=0} + \chi_0(x_1) B^+_2(\tfrac{x}{\varepsilon}) \, \partial_{x_2}v^+_0(x)\big|_{x_1=0} \Big),& x \in \Omega^+_\varepsilon,
\\[10pt]
    \displaystyle
v^-_0(x) +  \varepsilon \left(v^-_1(x) +
\chi_0(x_1)\, B^-_1(\tfrac{x}{\varepsilon}) \, \partial_{x_1}v^+_0(x)\big|_{x_1=0} + \chi_0(x_1) \, B^-_2(\tfrac{x}{\varepsilon}) \, \partial_{x_2}v^+_0(x)\big|_{x_1=0}
\right), & x \in \Omega^-.
                                     \end{array}
                                   \right.
\end{equation*}
\end{corollary}
The following inequalities follow from \eqref{two-term est}.
\begin{corollary}
\begin{gather}
  \left\|u_\varepsilon -  v_0\right\|_{L^2(\Omega_\varepsilon)} \le C_1\, \varepsilon,
 \\[4pt]
  \left\|\nabla u_\varepsilon - \begin{pmatrix}
  (1+\partial_{\xi_1}N_1) & \partial_{\xi_1}N_2
\\[2pt] \label{grad1}
 \partial_{\xi_2}N_1  & (1+\partial_{\xi_2}N_2)
\end{pmatrix}
\nabla v^+_0
- \chi_0\begin{pmatrix}
  \partial_{\xi_1}B^+_1 & \partial_{\xi_1}B^+_2
\\[2pt]
 \partial_{\xi_2}B^+_1  & \partial_{\xi_2}B^+_2
\end{pmatrix}
\nabla v^+_0|_{x_1=0}
 \right\|_{L^2(\Omega^+_\varepsilon)} \le C_2\, \varepsilon,
\\[4pt] \label{grad2}
\left\|\nabla u_\varepsilon - \nabla v^-_0
- \chi_0 \begin{pmatrix}
  \partial_{\xi_1}B^-_1 & \partial_{\xi_1}B^-_2
\\[2pt]
 \partial_{\xi_2}B^-_1  & \partial_{\xi_2}B^-_2
\end{pmatrix}
\nabla v^+_0|_{x_1=0}
 \right\|_{L^2(\Omega^-)} \le C_2\, \varepsilon,
\end{gather}
where $u_\varepsilon$ is the solution to problem \eqref{original Problem}, $v_0$ is the solution to the homogenized problem \eqref{Homo-problem},
$N_1$ and $N_2$ are the solutions to problems \eqref{cell1} and \eqref{cell2}, respectively, and $B_1$ and $B_2$ are the solutions to problems \eqref{Innner cell1} and \eqref{Innner cell2}, respectively.
\end{corollary}

Taking into account the exponential decrease of $B_1$ and $B_2,$ we derive the following inequalities from~\eqref{two-term est}.
\begin{corollary}\label{Cor-4-3}
For any domain $Q\subset \Omega^-$ such that $dist(\partial Q, {\mathcal{Z}}) >0,$ we have
$$
\left\|u_\varepsilon - v^-_0 \right\|_{H^1(Q)} \le C_1\, \varepsilon;
$$
and for any domain $Q\subset \Omega^+$ such that $dist(\partial Q, {\mathcal{Z}}) >0:$
$$
\left\|u_\varepsilon - v^+_0 - \varepsilon \sum_{i=1}^{2}N_i(\tfrac{x}{\varepsilon}) \, \partial_{x_i}v^+_0 \right\|_{H^1(Q\, \cap \,  \Omega^+_\varepsilon)} \le C_1\, \varepsilon.
$$
\end{corollary}

\section{Rapidly oscillating interface}\label{Sect-5}

In addition to the notation  in Sect.~\ref{Sec-2}, we introduce a function $\ell \in C^4(\Bbb R),$ which  is $1$-periodic, even, non-positive and $\ell(0) = 0.$
Clearly, the function $\ell(\frac{x_2}{\varepsilon}), \ x_2 \in \Bbb R,$ is $\varepsilon$-periodic. The curve
$x_1= \varepsilon \,  \ell(\tfrac{x_2}{\varepsilon}),$ $ x_2 \in [0, d],$
represents an oscillating interface  dividing the domain $\Omega $ into two subdomains
$$
\Omega^-_\varepsilon := \left\{x\in \Omega \colon \  x_1 < \varepsilon \,  \ell(\tfrac{x_2}{\varepsilon})\right\} \quad \text{and} \quad
 \Omega^{+,\varepsilon} := \left\{x\in \Omega \colon \  x_1 > \varepsilon \,  \ell(\tfrac{x_2}{\varepsilon})\right\}.
 $$
\begin{figure}[htbp]
  \centering
  \includegraphics[width=7cm]{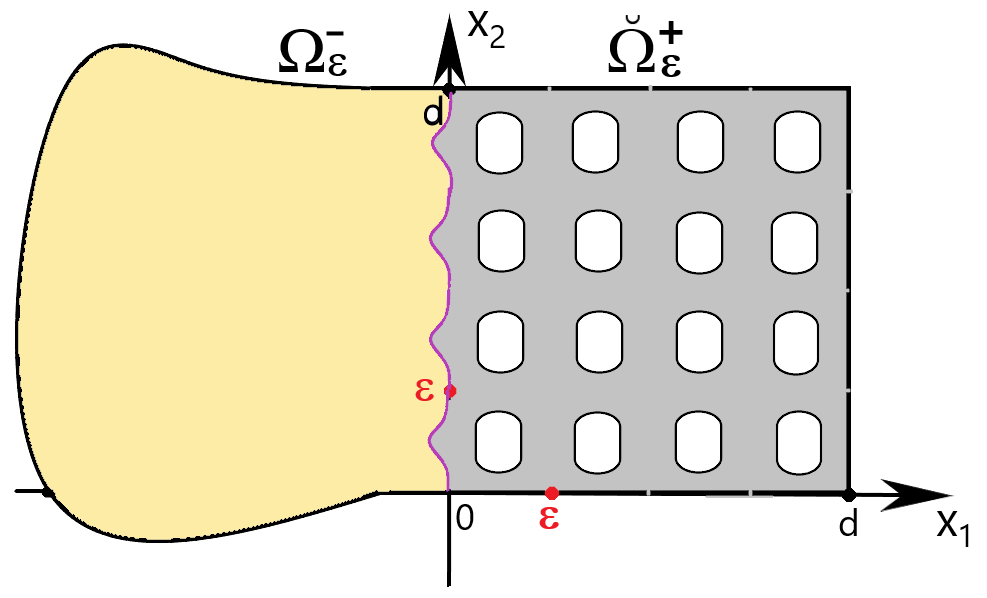}
  \caption{The partial perforated domain  $\Omega_\varepsilon$}\label{f2}
   \end{figure}
Then the perforated part is  $\breve{\Omega}^+_\varepsilon:= \Omega^{+,\varepsilon} \cap  \mathcal{Y}_{{\varepsilon}},$ and
partially perforated domain is defined as follows
$$
\Omega_\varepsilon := \Omega^-_\varepsilon   \cup \Lambda_\varepsilon  \cup \breve{\Omega}^+_{{\varepsilon}},
$$
where
$
\Lambda_\varepsilon := \left\{x\colon \ x_1= \varepsilon \,  \ell(\tfrac{x_2}{\varepsilon}), \quad x_2 \in [0, d]\right\}
$
is  the oscillating interface  (see Fig.~\ref{f2}). We assume that $\Lambda_\varepsilon$ does not intersects the hole boundaries~ $\partial G_\varepsilon.$

In $\Omega_\varepsilon$ we now consider the problem
\begin{equation}\label{new-problem}
\left\{
\begin{array}{lcl}
D^- \Delta_x u^-_{\varepsilon}(x) = f^-(x) \ \ \text{in} \ \ \Omega^-_\varepsilon, & &  \Delta_{ x}u^+_{\varepsilon}(x) = f^+(x) \ \ \text{in} \ \ \breve{\Omega}^+_{{\varepsilon}},
\\[2pt]
u^-_{\varepsilon}=0 \ \ \text{on} \ \ \Gamma^-,& &
u^+_{\varepsilon} = 0  \ \ \text{on} \ \ \Gamma^+,
\\[2pt]
 & & \nabla_x u^+_{\varepsilon} \cdot \vec{\nu}_\varepsilon =0 \ \ \text{on} \ \ \partial G_{\varepsilon},
\\[2pt]
u^-_{\varepsilon}   = u^+_{\varepsilon} \ \ \text{on} \ \ \Lambda_\varepsilon, & &
\\[2pt]
D^- \nabla_x u^-_{\varepsilon} \cdot \vec{\nu}_\varepsilon = \nabla_x u^+_{\varepsilon} \cdot \vec{\nu}_\varepsilon  \ \ \text{on} \ \ \Lambda_\varepsilon, & &
\end{array}
\right.
\end{equation}
where $D^-$ is a positive constant and the  unit normal to the interface $\Lambda_\varepsilon$ is defined as follows
\begin{equation}\label{normal}
  \vec{\nu}_\varepsilon :=  \left(\frac{1}{\sqrt{1+ |\ell^\prime(\xi_2)|^2}}, \ - \frac{\ell^\prime(\xi_2)}{\sqrt{1+ |\ell^\prime(\xi_2)|^2}}  \right)\bigg|_{\xi_2 = \frac{x_2}{\varepsilon}}, \quad x_2\in (0,d),
\end{equation}
(it is an outward normal with respect to $\Omega^-_\varepsilon).$ We assume that $f^\pm \in C^2_0(\Omega^\pm).$

It is obvious that for each fixed value of the parameter $\varepsilon$ there is  a weak unique solution
$$
u_\varepsilon(x) = \left\{
                                       \begin{array}{ll}
                                         u^-_{\varepsilon}(x), & x \in \Omega^-_\varepsilon,
 \\[2pt]
                                         u^+_{\varepsilon}(x), & x \in \breve{\Omega}^+_{\varepsilon},
                                       \end{array}
                                     \right.
$$
to problem \eqref{new-problem} and $\|u_\varepsilon\|_{H^1(\Omega_\varepsilon)} \le C_1.$ Moreover, by virtue of the  assumptions for $f^\pm$ and $\ell$ and based on the results obtained in \cite{Scheftel-1965,Shish-1961}, the solution is classical (see these papers for the definition). Furthermore,  $u^-_\varepsilon$ and $u^+_\varepsilon$ belong  to the H\"older spaces $C^{3,\mu}$ within the closures of the domains $\Omega^-_\varepsilon$ and  $\breve{\Omega}^+_{\varepsilon},$ respectively. Note that these results for corner points should be applied in the same manner as in  \S~\ref{par-3-3}.

As in these articles, one can establish Schauder estimates for the solution, as well as the corresponding estimates in H\"older norms.
Additionally, by repeating the relevant proofs in our case, it can be seen that  constants in these estimates estimates remain independent of
 $\varepsilon.$

 For example, in the proof of Theorem 1 \cite{Scheftel-1965}, to  locally straighten the surface $\Lambda_\varepsilon,$  one can use the mapping $y_1 = x_1 - \varepsilon \ell(\frac{x_2}{\varepsilon}),$ $y_2 = x_2$ in an $\varepsilon$-vicinity, since  the function $\ell(\frac{x_2}{\varepsilon})$ is $\varepsilon$-periodic. Clearly, its  Jacobian determinant equals 1, and  the line differential element
$\sqrt{1+ |\ell^\prime(\xi_2)|^2}|_{\xi_2 = \frac{x_2}{\varepsilon}}$ is uniformly bounded with respect to $\varepsilon.$ Thus,
\begin{equation}\label{est-solution}
  \|u^+_\varepsilon\|_{ C^1\big(\overline{\breve{\Omega}^+_{\varepsilon}}\big)} \ + \
\|u^-_\varepsilon\|_{ C^1\big(\overline{\Omega^-_\varepsilon}\big)} \le C_1.
\end{equation}

Now the  study focuses on constructing an asymptotic approximation ( as $\varepsilon \to 0)$  for the solution  while simultaneously establishing asymptotic estimates for the solution and its gradient,  with a particular focus on analyzing the impact of the oscillating interface on these results.

\begin{remark}
More general differential operators can be considered in problem  \eqref{new-problem} (see Remark~\ref{Rem-2-2}).
\end{remark}

\subsection{Construction of approximation}
In the perforated square $\Omega^+_\varepsilon$ the asymptotics is defined in the same way as in \S~\ref{subsect-3-1}.
The inner-layer asymptotics is  given in the domain $\Upsilon,$ which is now the union of the sets
$$
\Upsilon^- := \{\xi\colon \ \xi_1 < \ell(\xi_2), \ \ 0<\xi_2<1\}, \quad \Upsilon^{+} := \{\xi\colon \ \ell(\xi_2) < \xi_1, \ 0<\xi_2<1\} \bigcap
\Big(\bigcup_{k \in \Bbb N_0} \big( Y + (k, 0)\big)\Big),
$$
and $\lambda := \{\xi\colon \ \xi_1 = \ell(\xi_2), \ \ 0<\xi_2<1\}$  (see Fig.~\ref{f5}).

\begin{figure}[htbp]\label{F}
  \centering
  \includegraphics[width=10cm]{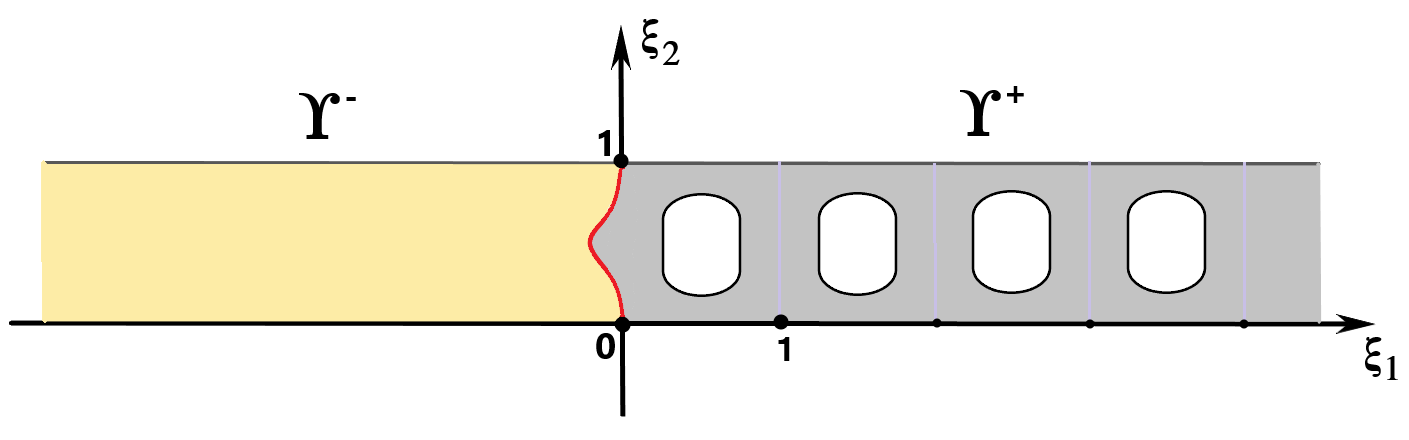}
  \caption{Partial perforated band-cell}\label{f5}
   \end{figure}

The following model problem is now considered:
\ find \ $ B(\xi) = \left\{
                     \begin{array}{l}
                       B^+(\xi), \  \xi \in \Upsilon^+, \\
                       B^-(\xi), \ \xi \in \Upsilon^-,
                     \end{array}
                   \right.
$
 \  that solves the problem
\begin{equation}\label{Innner cell-model+1}
\left\{
\begin{array}{l}
D^- \Delta_{\xi}B^-(\xi)  = F^-_0(\xi) +  \partial_{\xi_2}F^-_1(\xi) \ \  \text{in} \ \ \Upsilon^-,
\qquad
\Delta_{\xi}B^+(\xi)  = F^+_0(\xi) +  \partial_{\xi_2}F^+_1(\xi)
\ \  \text{in} \ \ \Upsilon^+,
\\[4pt]
\partial_{\vec{\nu}(\xi_2)}B^+(\xi) =  \nu_{2}(\xi_2)\,
F^+_1(\xi)  \ \  \text{on} \ \ \mathfrak{S}^+,
\\[4pt]
\big[B \big]_{\xi \in \lambda} = \Phi(\xi_2) + {q}  \quad  \text{and} \quad
\big[\partial_{\vec{\nu}_\xi}B \big]_{\xi \in \lambda}
 = \Psi(\xi_2)  + {J}\, \tau(\xi_2)  \quad \xi_2 \in (0,1),
\\[4pt]
B^\pm(\xi) \to 0 \ \  \text{as} \ \ \xi_1\to \pm \infty, \qquad  B^\pm(\xi)\ \  \text{are} \ 1\text{-periodic in} \ \xi_2,
\end{array}
\right.
\end{equation}
where $\big[B \big]_{\xi \in \lambda} := B^+(\xi) - B^-(\xi),$ \ $\big[\partial_{\vec{\nu}_\xi}B \big]_{\xi \in \lambda} := \big(\nabla_\xi B^+(\xi) - D^- \nabla_\xi B^-(\xi)\big) \cdot  \vec{\nu}_\xi,$ and
$$
\vec{\nu}_\xi  = \big(\nu_1(\xi_2), \, \nu_2(\xi_2)\big) = \left(\frac{1}{\sqrt{1+ |\ell^\prime(\xi_2)|^2}}, \ - \frac{\ell^\prime(\xi_2)}{\sqrt{1+ |\ell^\prime(\xi_2)|^2}}  \right), \quad \xi_2 \in [0, 1].
$$

Similarly to the proof of Theorem~\ref{Th-3-1}, using results from  \cite{Ole-Ios} and \cite[Chapt. I, \S 8]{Ole-Ios-Sha-1991}, we derive the theorem.
\begin{theorem}\label{Th-5-1}
   Let the right-hand sides $F^\pm_0, \ F^\pm_1, \ \Phi, \ \Psi, \  \tau$ in problem \eqref{Innner cell-model} be smooth functions in their domains of definition and  1-periodic in $\xi_2.$
Let $e^{\delta_0\, |\xi_1|} \, F^\pm_0 \in L^2(\Upsilon^\pm), \ \ e^{\delta_0\, |\xi_1|} \, F^\pm_1 \in L^2(\Upsilon^\pm)$ for some $\delta_0 > 0$
 and
\begin{equation}\label{cons+J}
 { J} = - \frac{1}{\int_{\lambda}\tau(\xi_2) \, dl_{\xi}} \left(\int_{\lambda}\Psi(\xi_2) \, dl_{\xi} + \int_{\Upsilon^\pm} F^\pm_0(\xi)\, d\xi \right).
\end{equation}

Then there exists a unique number ${q} \in \Bbb R$ and a unique  solution to problem \eqref{Innner cell-model+1}  with the following differential asymptotics
\begin{equation*}
  B(\xi) = \mathcal{O}\big(e^{-\delta\, |\xi_1|}\big) \ \ \text{as} \ \ |\xi_1| \to \infty \quad (\delta >0).
\end{equation*}

In addition,
\begin{itemize}
  \item if $F^\pm_0, \ \Phi, \ \Psi, \ \tau$ are odd in $\xi_2$ and $F^\pm_1$ are even in $\xi_2,$ then the solution $B$ is odd in $\xi_2$
and
$$
{q} =0 \quad \text{and} \quad {J}=0;
$$
\item if $F^\pm_0, \ \Phi, \ \Psi, \ \tau$ are even in $\xi_2$ and $F^\pm_1$ are odd in $\xi_2,$ then the solution $B$ is even  in $\xi_2.$
\end{itemize}
\end{theorem}

The first terms of the  inner-layer asymptotics are solutions to the following problems:
\begin{equation}\label{Innner cell1+}
\left\{
\begin{array}{l}
\Delta_{\xi}B^\pm_1(\xi) = 0 \ \  \text{in} \ \ \Upsilon^\pm, \qquad \partial_{\vec{\nu}_\xi} B^+_1(\xi) = 0\ \  \text{on} \ \ \mathfrak{S}^+,
\\[4pt]
\big[ B_1 \big]_{\xi\in \lambda}  =  - N_1(\ell(\xi_2), \xi_2) + q_1, \quad \xi_2 \in (0,1),
\\[4pt]
\big[\partial_{\vec{\nu}_\xi}B_1 \big]_{\xi \in \lambda}
  =  - \big(\nabla_{\xi}N_1 \cdot \vec{\nu}_\xi\big)\big|_{\xi \in \lambda} + \tilde{J}_1 \, \nu_1(\xi_2), \quad \xi_2 \in (0,1),
\\[4pt]
B^\pm_1(\xi) \to 0 \ \  \text{as} \ \ \xi_1\to \pm \infty, \qquad  B^\pm_1(\xi)\ \  \text{are} \ 1\text{-periodic in} \ \xi_2,
\end{array}
\right.
\end{equation}
and
\begin{equation}\label{Innner cell2+}
\left\{
\begin{array}{l}
\Delta_{\xi}B^\pm_2(\xi) = 0 \ \  \text{in} \ \ \Upsilon^\pm, \qquad \partial_{\vec{\nu}_\xi} B^+_2(\xi) = 0\ \  \text{on} \ \ \mathfrak{S}^+,
\\[4pt]
\big[ B_1 \big]_{\xi\in \lambda}  =  - N_2(\ell(\xi_2), \xi_2), \quad \xi_2 \in (0,1),
\\[4pt]
\big[\partial_{\vec{\nu}_\xi}B_2 \big]_{\xi \in \lambda}
  =  - \big(\nabla_{\xi}N_2 \cdot \vec{\nu}_\xi\big)\big|_{\xi \in \lambda} + (D^- -1) \, \nu_2(\xi_2), \quad \xi_2 \in (0,1),
\\[4pt]
B^\pm_1(\xi) \to 0 \ \  \text{as} \ \ \xi_1\to \pm \infty, \qquad  B^\pm_1(\xi)\ \  \text{are} \ 1\text{-periodic in} \ \xi_2,
\end{array}
\right.
\end{equation}
where $N_1$ and $N_2$ are smooth $1$-periodic solutions to problems  \eqref{cell1} and \eqref{cell2}, respectively.

\begin{remark}
  It should be noted that if $\ell \equiv 0,$ then for  $D^- =1$  problem \eqref{Innner cell2+} coincides with problem~\eqref{Innner cell2},  and   problem \eqref{Innner cell1+}  with  problem \eqref{Innner cell1}, where $J_1 = \tilde{J}_1 - 1.$
\end{remark}

Taking into account the symmetry properties of  $N_1$ and $N_2$ (see \eqref{sym1}) and the evenness  of the function~$\ell,$ we conclude that
$\big(\nabla_{\xi}N_1 \cdot \vec{\nu}_\xi\big)\big|_{\xi \in \lambda}$ and $\nu_1(\xi_2)$ are even in $\xi_2,$ and
$\big(\nabla_{\xi}N_2 \cdot \vec{\nu}_\xi\big)\big|_{\xi \in \lambda}$ and $\nu_2(\xi_2)$ are odd in $\xi_2.$

Applying Theorem \ref{Th-5-1} to problem  \eqref{Innner cell2+}, we see that the solvability condition is automatically fulfilled, which  means that there exists a unique solution that is odd in $\xi_2$ and decreases exponentially  at infinity. The solvability condition for problem \eqref{Innner cell1+} reads as follows
\begin{equation}\label{const-J}
  \tilde{J}_1 = \int_{\lambda}\nabla_{\xi}N_1 \cdot \vec{\nu}_\xi  \, dl_{\xi} .
\end{equation}
Thus,  there exists a unique number $q_1 \in \Bbb R$ and a unique  solution to problem \eqref{Innner cell1+}, which is even  in $\xi_2$ and decreases exponentially  at infinity.

\begin{proposition}\label{Prop-5-1}
  The number
\begin{equation}\label{J-2}
\tilde{J}_1 = - 1\,  +  \upharpoonleft\!\! Y\!\!\upharpoonright  {h_{11}},
\end{equation}
where $h_{11}$ is determined by formula \eqref{ell1}.
\end{proposition}
\begin{proof} Due to the 1-periodicity of the solution $N_1$ to problem \eqref{cell1} we  have that
$$
 \partial_{\xi_1}\big(1 + \partial_{\xi_1}N_1\big) + \partial^2_{\xi_2}N_1 =0 \quad \text{in} \ \ \tilde{Y} : = Y \cup \{\xi\colon \  \ell(\xi_2) < \xi_1 \le 0,  \ \ \xi_2 \in (0, 1)\}.
$$
Then we multiply this equation by the test function
$$
\phi(\xi) = \left\{
              \begin{array}{ll}
                -\xi_1 +1, & \xi \in Y, \\
                1, & \xi \in \{\xi\colon \  \ell(\xi_2) < \xi_1 \le  0,  \ \ \xi_2 \in (0, 1)\},
              \end{array}
            \right.
$$
and integrate by parts in $\tilde{Y}.$ As a result, we get
\begin{equation}\label{J-1}
  0  =   - \int_{\lambda}\nu_1(\xi) \, dl_{\xi} - \int_{\lambda}\nabla_{\xi}N_1 \cdot \vec{\nu}_\xi  \, dl_{\xi}
+ \int_{Y} \big(1 + \partial_{\xi_1}N_1\big) \, d\xi .
\end{equation}
Since
$$
\int_{\lambda}\nu_1(\xi) \, dl_{\xi} = \int_{0}^{1} \frac{1}{\sqrt{1+ |\ell^\prime(\xi_2)|^2}} \, \sqrt{1+ |\ell^\prime(\xi_2)|^2} \, d\xi_2 = 1,
$$
relation \eqref{J-2} follows from  \eqref{J-1}.
\end{proof}

Let  $ v_0(x) = \left\{
                     \begin{array}{ll}
                        v^+_0(x), & x \in \Omega^+, \\
                        v^-_0(x), & x  \in \Omega^-,
                     \end{array}
                   \right.
$
 be a solution to the problem
\begin{equation}\label{Homo-problem+1}
\left\{
\begin{array}{lcl}
D^- \, \Delta_x v^-_0(x) = f^-(x) \ \ \text{in} \ \ \Omega^- , & &  \upharpoonleft\!\! Y\!\!\upharpoonright {h_{11}} \, \partial^2_{x_1}v^+_0(x) \  + \  \upharpoonleft\!\! Y\!\!\upharpoonright {h_{22}}\,  \partial^2_{x_2}v^+_0(x) = \ \upharpoonleft\!\! Y\!\!\upharpoonright f^+(x) \ \ \text{in} \ \ \Omega^+ ,
\\[2pt]
v^-_0(x)=0 \ \ \text{on} \ \ \Gamma^-,& &
v^+_0(x) = 0  \ \ \text{on} \ \ \Gamma^+,
\\[2pt]
v^-_0(0,x_2)   = v^+_0(0,x_2) \ \ \text{on} \ \ {\mathcal{Z}}, & &
D^-\, \partial_{x_1}v^-_0(x) =\  \upharpoonleft\!\! Y\!\!\upharpoonright  {h_{11}} \, \partial_{x_1}v^+_0(x) \ \ \text{on} \ \ {\mathcal{Z}},
\end{array}
\right.
\end{equation}
which is  a \emph{homogenized problem} for problem \eqref{new-problem}.
In the same way as in \S~\ref{par-3-3}, based on the results of
\cite{Scheftel-1965,Shish-1961}) and the assumptions about the smoothness of $f^\pm$ and $\ell,$ we conclude that
$
v_0^\pm \in C^{3, \mu}\big(\overline{\Omega^\pm} \big).
$
Considering the smoothness of $N_1, N_2$ and $\ell$ with the same arguments, one finds that the inclusions
$
B_i^\pm \in C^{2, \mu}\big(\overline{\Upsilon^\pm} \big),
$
$i\in \{1, 2\},$ hold for the solutions to problems \eqref{Innner cell1+} and
\eqref{Innner cell2+}.

Using the solutions to problems \eqref{Homo-problem+1}, \eqref{Innner cell1+}, \eqref{Innner cell2+}, \eqref{cell1} and \eqref{cell2}, we construct the approximation
\begin{equation}\label{app1}
  U_\varepsilon(x) :=
\left\{
\begin{array}{ll} \displaystyle
v^+_0(x) + \varepsilon \sum_{i=1}^2 \Big(N_i(\tfrac{x}{\varepsilon}) \,  \partial_{x_i}v^+_0(x)
 + \chi_0(x_1)  B^+_i(\tfrac{x}{\varepsilon}) \, \partial_{x_i}v^+_0(x)\big|_{x_1=0}\Big), &  x \in \Omega^+_\varepsilon,
\\ \displaystyle
\tilde{v}^+_0(x) + \varepsilon \sum_{i=1}^2 \Big(N_i(\tfrac{x}{\varepsilon}) \,  \partial_{x_i}\tilde{v}^+_0(x)
 + \chi_0(x_1)  B^+_i(\tfrac{x}{\varepsilon}) \, \partial_{x_i}v^+_0(x)\big|_{x_1=0}\Big), &  x \in  \breve{\Omega}^+_{{\varepsilon}} \setminus \Omega^+_\varepsilon,
\\ \displaystyle
v^-_0(x) + \varepsilon\,  \chi_0(x_1) \, \sum_{i=1}^2   B^-_i(\tfrac{x}{\varepsilon}) \, \partial_{x_i}v^+_0(x)\big|_{x_1=0}, &  x \in \Omega^-_\varepsilon,
\end{array}
\right.
\end{equation}
where the cut-off function $\chi_0$ is defined in \eqref{cut-off}, $\tilde{v}^+_0$ is the $C^3$-extension of ${v}^+_0$ from $\Omega^+_\varepsilon$ onto
$\breve{\Omega}^+_{{\varepsilon}} \setminus \Omega^+_\varepsilon.$
Obviously, $U^+_\varepsilon \in C^{2, \mu}\big(\overline{\breve{\Omega}^+_{\varepsilon}} \big),$ $ U^-_\varepsilon \in C^{2, \mu}\big(\overline{\Omega^-_{\varepsilon}} \big)$ and
\begin{equation}\label{est-appr}
  \|U^+_\varepsilon\|_{ C^1\big(\overline{\breve{\Omega}^+_{\varepsilon}}\big)} \ + \
\|U^-_\varepsilon\|_{ C^1\big(\overline{\Omega^-_\varepsilon}\big)} \le C_1.
\end{equation}

\subsection{Justification} In virtute of the symmetry properties of $N_1, N_2, B_1, B_2$, similarly as in Sect.~\ref{Sect_3}, we check  that $U_\varepsilon\big|_{\partial\Omega} =0.$ Now we calculate
\begin{align}\label{cal-1}
 \Delta_x \Big(v^+_0 + \varepsilon \sum_{i=1}^2 N_i \,  \partial_{x_i}v^+_0\Big) & = \sum_{i, j =1}^{2}\big(\delta_{i, j} + 2 \partial_{\xi_i}N_j\big) \, \partial^2_{x_i x_j}v^+_0 +  \varepsilon \sum_{i=1}^2 N_i \,  \partial_{x_i}\Delta_x v^+_0 \notag
\\
 &  \stackrel{\eqref{cell5}}{=}   f^+  - \sum_{i, j =1}^{2} \Delta_\xi N_{i j}(\xi)\big|_{\xi= \frac{x}{\varepsilon}} \partial^2_{x_i x_j}v_0^+  +   \varepsilon \sum_{i=1}^2 N_i \,  \partial_{x_i}\Delta_x v^+_0 \notag
\\
& =  f^+ + \varepsilon\, F^+_0 + \varepsilon \sum_{k=1}^{2} \partial_{x_k}F^+_k \quad \text{in} \ \  \Omega^+_\varepsilon,
\end{align}
where $\{N_{i j}\}$ are solutions to problems \eqref{cell5},
\begin{gather}\label{F-0}
  F^+_0(x; \varepsilon) = \sum_{i=1}^2 N_i(\tfrac{x}{\varepsilon}) \,  \partial_{x_i}\Delta_x v^+_0(x) + \sum_{i, j, k=1}^2 \partial_{\xi_k} N_{i j}(\xi)\big|_{\xi= \frac{x}{\varepsilon}} \,  \partial^3_{x_k x_i x_j}v^+_0(x), \quad x\in \Omega^+_\varepsilon,
\\
\label{F-k}
  F^+_k(x; \varepsilon) =  - \sum_{i, j =1}^2 \partial_{\xi_k} N_{i j}(\xi)\big|_{\xi= \frac{x}{\varepsilon}} \,  \partial^2_{x_i x_j}v^+_0(x), \quad x\in \Omega^+_\varepsilon, \quad k \in \{1, 2\}.
\end{gather}
It follows from \eqref{F-0} and \eqref{F-k} that
\begin{equation}\label{e-1}
  \sum_{i=0}^{2} \max_{x\in \overline{\Omega^+_\varepsilon}} \big| F^+_i(x; \varepsilon) \big| \le C_1.
\end{equation}
The normal derivative
$$
\nabla_x\Big(v^+_0 + \varepsilon \sum_{i=1}^2 N_i \,  \partial_{x_i}v^+_0\Big) \cdot \vec{\nu}_\varepsilon = \varepsilon
\sum_{k=1}^{2} F^+_k \, \nu_k(\tfrac{x}{\varepsilon}) \quad \text{on} \ \ \partial G_{\varepsilon}.
$$

In the same way we verify that
\begin{equation}\label{cal}
  \Delta_x \Big(\tilde{v}^+_0 + \varepsilon \sum_{i=1}^2 N_i \,  \partial_{x_i}\tilde{v}^+_0\Big) =
{\widehat{\mathcal{H}}}\,\tilde{v}^+_0 + \varepsilon\, \tilde{F}^+_0 + \varepsilon \sum_{k=1}^{2} \partial_{x_k}\tilde{F}^+_k \quad \text{in} \ \ \breve{\Omega}^+_{{\varepsilon}} \setminus  \Omega^+_\varepsilon,
\end{equation}
where  ${\widehat{\mathcal{H}}} = {h_{11}} \, \partial^2_{x_1}  +  {h_{22}}\,  \partial^2_{x_2},$ \,  the functions
$\{\tilde{F}^+_k\}_{k=0}^2$ are defined by formulas \eqref{F-0} and \eqref{F-k}, in which the extension $\tilde{v}^+_0$ replaces $v_0^+,$ and an estimate similar to \eqref{e-1} is valid.

Since ${\widehat{\mathcal{H}}}\,\tilde{v}^+_0\big|_{x_1 =0} = {\widehat{\mathcal{H}}}\,{v}^+_0\big|_{x_1 =0} =0,$ by Taylor's formula we get that ${\widehat{\mathcal{H}}}\,\tilde{v}^+_0 = \mathcal{O}(\varepsilon)$ as $\varepsilon \to 0$ in $\breve{\Omega}^+_{{\varepsilon}} \setminus  \Omega^+_\varepsilon.$ Further we will assume that ${\widehat{\mathcal{H}}}\,\tilde{v}^+_0$ is included into  $\tilde{F}^+_0.$

Similarly, we also calculate
\begin{gather*}
 \varepsilon \, \Delta_x \Big(\sum_{i=1}^2 B^\pm_i(\tfrac{x}{\varepsilon}) \, \partial_{x_i}v^+_0(x)\big|_{x_1=0} \Big) =
\varepsilon\, \mathcal{F}^\pm_0 + \varepsilon \sum_{k=1}^{2} \partial_{x_k}\mathcal{F}^\pm_k \quad \text{in} \ \ \breve{\Omega}^+_{{\varepsilon}} \cup \Omega^-_\varepsilon,
\\
\varepsilon \ \nabla_x\Big(\sum_{i=1}^2 B^+_i(\tfrac{x}{\varepsilon}) \, \partial_{x_i}v^+_0(x)\big|_{x_1=0} \Big) \cdot \vec{\nu}_\varepsilon = \varepsilon
\sum_{k=1}^{2} \mathcal{F}^+_k \, \nu_k(\tfrac{x}{\varepsilon}) \quad \text{on} \ \ \partial G_{\varepsilon}.
\end{gather*}
Note  that  to calculate these we use differential equations for the coefficients  $\{B_{i j}\},$ which are the same
as in \eqref{B-alpha} for $|\alpha| =2.$ The functions $\{\mathcal{F}^\pm_k\}_{k=0}^2$ satisfy  an estimate similar to \eqref{e-1} and  are exponentially small on the support of $\chi^\prime_0(x_1).$

Now let us see what happens at the interface. For the difference $\big[U_\varepsilon\big]_{x \in \Lambda_\varepsilon},$ based on the jump conditions for the solutions of problems \eqref{Innner cell1+}, \eqref{Innner cell2+} and \eqref{Homo-problem+1}, we obtain the following results:
\begin{align*}
  \big[U_\varepsilon\big]_{x \in \Lambda_\varepsilon} = & \ \tilde{v}^+_0 + \varepsilon \sum_{i=1}^2 \Big(N_i(\tfrac{x}{\varepsilon}) \,  \partial_{x_i}\tilde{v}^+_0(x)
 +  B^+_i(\tfrac{x}{\varepsilon}) \, \partial_{x_i}v^+_0(x)\big|_{x_1=0}\Big)
 - \bigg(v^-_0 + \varepsilon\,   \sum_{i=1}^2   B^-_i(\tfrac{x}{\varepsilon}) \, \partial_{x_i}v^+_0(x)\big|_{x_1=0}\bigg)
\\
= & \  \tilde{v}^+_0(\varepsilon \ell(\tfrac{x_2}{\varepsilon}), x_2) - v^-_0(\varepsilon \ell(\tfrac{x_2}{\varepsilon}), x_2)  + \varepsilon  \sum_{i=1}^2 \Big(N_i(\xi) - [B_i(\xi)]_{\xi \in \lambda}\Big)\Big|_{\xi =\frac{x}{\varepsilon}} \, \partial_{x_i}v^+_0(x)\big|_{x_1=0}
\\
& \  -  \varepsilon  \sum_{i=1}^2N_i(\tfrac{x}{\varepsilon}) \Big(\partial_{x_i}\tilde{v}^+_0(x)\big|_{x_1 =\varepsilon \ell(\frac{x_2}{\varepsilon})} - \partial_{x_i}v^+_0(x)\big|_{x_1=0}\Big)
\\
=& \  \big( \tilde{v}^+_0(\varepsilon \ell(\tfrac{x_2}{\varepsilon}), x_2) - v_0^+(0, x_2)\big) -
\big( {v}^-_0(\varepsilon \ell(\tfrac{x_2}{\varepsilon}), x_2) - v_0^-(0, x_2)\big) + \varepsilon  q_1  \, \partial_{x_1}v^+_0(x)\big|_{x_1=0}
\\
& \ - \varepsilon  \sum_{i=1}^2N_i(\tfrac{x}{\varepsilon}) \Big(\partial_{x_i}\tilde{v}^+_0(x)\big|_{x_1 =\varepsilon \ell(\frac{x_2}{\varepsilon})} - \partial_{x_i}v^+_0(x)\big|_{x_1=0}\Big) =: \varepsilon \, \Psi(x; \varepsilon), \quad x \in \Lambda_\varepsilon,
\end{align*}
where
\begin{equation}\label{e-2}
\max_{x_2 \in [0,d]} \left|\Psi\big(\varepsilon \ell(\tfrac{x_2}{\varepsilon}),  x_2; \varepsilon\big)\right| \le C_2.
\end{equation}

Taking into account \eqref{J-2} and second transmission conditions for the solutions of problems \eqref{Innner cell1+}, \eqref{Innner cell2+} and \eqref{Homo-problem+1}, we find  the jump of the normal derivative
\begin{align*}
 [\partial_{\vec{\nu}_\varepsilon} U]_{x\in \Lambda_\varepsilon} :=  & \ \big(\nabla_x U^+_{\varepsilon} -  D^- \nabla_x U^-_{\varepsilon}\big) \cdot \vec{\nu}_\varepsilon
\\
 =  & \ \big(\nabla_x \tilde{v}^+_0 -  D^- \nabla_x v^-_0\big)\big|_{x\in \Lambda_\varepsilon} \cdot \vec{\nu}_\varepsilon
  + \  \sum_{i=1}^2 \Big(\big(\nabla_{\xi} N_i(\xi)\big|_{\xi \in \lambda} + \big[\partial_{\vec{\nu}_\xi}B_i\big]_{\xi \in \lambda}\big) \cdot \vec{\nu}_\xi\Big)\Big|_{\xi =\frac{x}{\varepsilon}} \, \partial_{x_i}v^+_0(x)\big|_{x_1=0}
\\
& + \  \sum_{i=1}^2 \Big(\big(\nabla_{\xi} N_i(\xi)\big|_{\xi \in \lambda} \cdot \vec{\nu}_\xi\Big)\Big|_{\xi =\frac{x}{\varepsilon}} \, \Big(\partial_{x_i}\tilde{v}^+_0(x)\big|_{x_1 =\varepsilon \ell(\frac{x_2}{\varepsilon})} - \partial_{x_i}v^+_0(x)\big|_{x_1=0}  \Big) + \mathcal{O}(\varepsilon)
\\
= & \ \big(\nabla_x {v}^+_0(x)|_{x_1 =0} -  D^- \nabla_x v^-_0(x)|_{x_1 =0}\big) \cdot \vec{\nu}_\varepsilon
\\
& + \ \tilde{J}_1 \, \nu_1 \, \partial_{x_1}v^+_0(x)|_{x_1=0} + (D^- -1) \, \nu_2 \, \partial_{x_2}v^+_0(x)|_{x_1=0} + \mathcal{O}(\varepsilon)
\\
=& \ \Big((1+ \tilde{J}_1) \, \partial_{x_1}v^+_0|_{x_1=0}  - D^- \, \partial_{x_1}v^-_0|_{x_1=0}\Big) \, \nu_1
+
\Big(\partial_{x_2}v^+_0|_{x_1=0}  -  \partial_{x_2}v^-_0|_{x_1=0}\Big) \, \nu_2 + \mathcal{O}(\varepsilon)
\\
  = & \ \mathcal{O}(\varepsilon) =: \varepsilon \, \Phi(x; \varepsilon), \quad x \in \Lambda_\varepsilon,
\end{align*}
where
\begin{equation}\label{e-3}
\max_{x_2 \in [0,d]} \left|\Phi\big(\varepsilon \ell(\tfrac{x_2}{\varepsilon},  x_2; \varepsilon\big)\right| \le C_3.
\end{equation}

Summarising the above calculations, we can state that the difference $W_\varepsilon := U_\varepsilon - u_\varepsilon$ satisfies the following relations:
 \begin{equation}\label{results-problem}
\left\{
\begin{array}{llll}
D^- \Delta_x W^-_\varepsilon =
\varepsilon\, \mathcal{F}^-_0 + \varepsilon \sum_{k=1}^{2} \partial_{x_k}\mathcal{F}^-_k
 & \text{in} \ \ \Omega^-_\varepsilon,
& \quad
  \Delta_{ x} W^+_\varepsilon = \varepsilon\,
\mathfrak{F}^+_0 + \varepsilon \sum_{k=1}^{2} \partial_{x_k}\mathfrak{F}^+_k & \text{in} \ \ \breve{\Omega}^+_{{\varepsilon}},
\\[4pt]
W^+_\varepsilon =0  & \text{on} \ \ \partial\Omega, & \quad
\partial_{\vec{\nu}_\varepsilon} W_\varepsilon  = \varepsilon
\sum_{k=1}^{2} \mathfrak{F}^+_k \, \nu_k(\tfrac{x}{\varepsilon}) & \text{on} \ \ \partial G_{\varepsilon},
\\[4pt]
\big[W_\varepsilon\big]_{\Lambda_\varepsilon} = \varepsilon \, \Psi(x; \varepsilon), & x \in \Lambda_\varepsilon,
& \quad
\left[\partial_{\vec{\nu}_\varepsilon}W_\varepsilon\right]_{\Lambda_\varepsilon}
= \varepsilon \, \Phi(x; \varepsilon), & x \in \Lambda_\varepsilon,
\end{array}
\right.
\end{equation}
where the functions $\mathfrak{F}^+_k = F^+_k + \mathcal{F}_k^+, \ k \in \{0,1, 2\}.$

Multiplying the differential equations by $W_\varepsilon^+ := W_\varepsilon\big|_{\breve{\Omega}^+_\varepsilon}$ and $W_\varepsilon^- := W_\varepsilon\big|_{\Omega^-_\varepsilon},$ respectively, then  integrating by parts  and taking into account the boundary and interface relations, we get
\begin{multline}\label{main-rel}
D^- \int_{\Omega^-_\varepsilon}|\nabla_x W^-_\varepsilon|^2 dx + \int_{\breve{\Omega}^+_\varepsilon}|\nabla_x W^+_\varepsilon|^2 dx = \varepsilon \bigg(-\int_{\Omega^-_\varepsilon} \mathcal{F}^-_0 \,  W^-_\varepsilon\, dx  - \int_{\breve{\Omega}^+_\varepsilon} \mathfrak{F}^+_0 \,  W^+_\varepsilon\, dx
 \\
+ \sum_{k=1}^{2}\Big(\int_{\Omega^-_\varepsilon} \mathcal{F}^-_k \,  \partial_{x_k} W^-_\varepsilon\, dx  + \int_{\breve{\Omega}^+_\varepsilon} \mathfrak{F}^+_k \,  \partial_{x_k}W^+_\varepsilon\, dx
+ \int_{\Lambda_\varepsilon} \big(\mathfrak{F}^+_k \,  \nu_k \, W^+_\varepsilon - \mathcal{F}^-_k \, \nu_k \, W^-_\varepsilon\big) \, dl_x\Big) \bigg)
\\
-  \int_{\Lambda_\varepsilon} \big( \partial_{\vec{\nu}_\varepsilon}W^+_\varepsilon  \, W^+_\varepsilon -   D^- \, \partial_{\vec{\nu}_\varepsilon}W^-_\varepsilon \,  W^-_\varepsilon\big) \, dl_x.
\end{multline}
The integrand in the last integral can be rewritten as follows
\begin{equation}\label{main+1}
  \big(\partial_{\vec{\nu}_\varepsilon}W^+_\varepsilon -    D^- \, \partial_{\vec{\nu}_\varepsilon}W^-_\varepsilon\big) \, W^+_\varepsilon +  D^- \, \partial_{\vec{\nu}_\varepsilon}W^-_\varepsilon \, \big(W^+_\varepsilon -  W^-_\varepsilon\big) = \varepsilon \, \big(\Phi \, W^+_\varepsilon +
D^- \, \partial_{\vec{\nu}_\varepsilon}W^-_\varepsilon \, \Psi \big) \quad \text{on} \ \ \Lambda_\varepsilon.
\end{equation}
Taking into account the estimates for the residuals  $\{\mathcal{F}^-_k, \, \mathfrak{F}^+_k\}_{k=0}^2,$ $\Psi$ and $\Phi$ (see \eqref{e-1}, \eqref{e-2},
\eqref{e-3}) and also estimates \eqref{est-appr} and \eqref{est-solution}, it follows from \eqref{main-rel} and \eqref{main+1}
$$
\int_{\Omega^-_\varepsilon}|\nabla_x W^-_\varepsilon|^2 dx + \int_{\breve{\Omega}^+_\varepsilon}|\nabla_x W^+_\varepsilon|^2 dx \le  C \, \varepsilon.
$$
This proves the statement.
\begin{theorem}\label{Th-5-2}
There exist positive constants $C$ and $\varepsilon$ such that for all values of $\varepsilon \in (0, \varepsilon_0)$ it holds
\begin{equation}\label{main result}
  \left\|u^-_\varepsilon - U^-_\varepsilon \right\|_{H^1(\Omega^-_\varepsilon)} + \left\|u^+_\varepsilon - U^+_\varepsilon \right\|_{H^1(\breve{\Omega}^+_\varepsilon)} \le C \, \varepsilon^{\frac{1}{2}},
\end{equation}
where $u_\varepsilon^\pm$ is the solution to problem \eqref{new-problem}, and $U^\pm_\varepsilon$ is the approximation function defined in \eqref{app1}.
 \end{theorem}
The following inequalities follow from \eqref{main result}.
\begin{corollary}
\begin{gather}
  \left\|u^-_\varepsilon -  v^-_0\right\|_{L^2(\Omega^-_\varepsilon)} +  \left\|u^+_\varepsilon -  v^+_0\right\|_{L^2(\Omega^+_\varepsilon)} \le C_1\, \varepsilon^{\frac{1}{2}},
 \\[4pt]
  \left\|\nabla u^+_\varepsilon - \begin{pmatrix}
  (1+\partial_{\xi_1}N_1) & \partial_{\xi_1}N_2
\\[2pt] \label{grad1}
 \partial_{\xi_2}N_1  & (1+\partial_{\xi_2}N_2)
\end{pmatrix}
\nabla v^+_0
- \chi_0\begin{pmatrix}
  \partial_{\xi_1}B^+_1 & \partial_{\xi_1}B^+_2
\\[2pt]
 \partial_{\xi_2}B^+_1  & \partial_{\xi_2}B^+_2
\end{pmatrix}
\nabla v^+_0|_{x_1=0}
 \right\|_{L^2(\Omega^+_\varepsilon)} \le C_2\, \varepsilon^{\frac{1}{2}},
\\[4pt] \label{grad2+}
\left\|\nabla u^-_\varepsilon - \nabla v^-_0
- \chi_0 \begin{pmatrix}
  \partial_{\xi_1}B^-_1 & \partial_{\xi_1}B^-_2
\\[2pt]
 \partial_{\xi_2}B^-_1  & \partial_{\xi_2}B^-_2
\end{pmatrix}
\nabla v^+_0|_{x_1=0}
 \right\|_{L^2(\Omega^-_\varepsilon)} \le C_2\, \varepsilon^{\frac{1}{2}}.
\end{gather}
\end{corollary}

A consequence similar to Corollary~\ref{Cor-4-3}  holds, but only with an estimate of order $\mathcal{O}(\varepsilon^{\frac{1}{2}}).$

\section{Closing remarks}\label{Sect-6}

{\bf 1.} The results obtained in Section~\ref{Sect-5} do not  indicate  any influence of the interface microstructure in the homogenized problem \eqref{Homo-problem+1}. This aligns with the findings in \cite{Don-Piat-2010,Don-2019}  concerning imperfect contacts at interfaces in the case when the oscillation amplitude is of order $\mathcal{O}(\varepsilon)$  (the parameter $\kappa = 1)$ and the proportionality coefficient $\gamma$ in the solution jump is positive. The shape of the interface contributes to the homogenized problem only in the case where $\gamma = 0$
(see  \cite[Remark 4.2]{Don-Piat-2010}). 

However, in our case we observe a clear influence of the interface microstructure through the inner layer asymptotics discovered in this  work (see problems \eqref{Innner cell1+} and \eqref{Innner cell2+}). 
 The obtained results demonstrate  the advantage of the two-scale expansion approach over other methods which prove only convergence.
For instance, estimates \eqref{grad2} and  \eqref{grad2+} reveal  the rapid oscillatory nature of the solution’s gradient near the interface in the non-perforated region -- an essential structural detail that the $L^2$-limit fails to capture. This fundamental distinction underscores the deeper analytical insight provided by the two-scale expansion method, making it indispensable for accurately describing fine-scale behaviors in homogenization problems.

 \smallskip
 {\bf 2.}
 The shape of the interface in our case will manifest itself in the homogenized problem when we consider  the following second conjugation condition
\begin{equation}\label{new-second-tr}
  D^- \nabla_x u^-_{\varepsilon} \cdot \vec{\nu}_\varepsilon = \nabla_x u^+_{\varepsilon} \cdot \vec{\nu}_\varepsilon  + \Theta(x_2, \tfrac{x_2}{\varepsilon})
\ \ \text{on} \ \ \Lambda_\varepsilon,
\end{equation}
where $\Theta(x_2, \xi_2), \ x_2\in [0, d], \ \xi_2 \in [0,1],$ is $1$-periodic given function in $\xi_2,$ which belongs to the space $C^2([0,d]\times [0,1])$ and has a compact support in $(0, d).$
  Conditions of this type  can be used in physics and engineering to describe interactions where coupling effects exist and have measurable strength.

 The corresponding second conjugation condition in the homogenized problem \eqref{Homo-problem+1}
will then be as follows
\begin{equation}\label{new-hom-coup}
  D^-\, \partial_{x_1}v^-_0(x)\big|_{x_1 =0} =\  \upharpoonleft\!\! Y\!\!\upharpoonright  {h_{11}} \, \partial_{x_1}v^+_0(x)\big|_{x_1 =0} + \widehat{\Theta}(x_2), \quad x_2 \in (0,d),
\end{equation}
where
$$
\widehat{\Theta}(x_2) = \int_{0}^{1}\Theta(x_2, \xi_2) \, \sqrt{1+ |\ell^\prime(\xi_2)|^2} \, d\xi_2.
$$

The approximation to the solution in this is given by formula \eqref{app1}, where $v_0^\pm$ is now a solution to problem  \eqref{Homo-problem+1}
with the second transmission condition \eqref{new-hom-coup}.
To justify this asymptotic approximation and to obtain the same estimates as in Theorem~\ref{Th-5-2}, it is necessary to use  the inequality
$$
\bigg|\int_{\Lambda_\varepsilon} \Theta(x_2, \tfrac{x_2}{\varepsilon}) \, \phi \,  \psi \, dl_x - \int_{0}^{d} \widehat{\Theta}(x_2) \, \phi(0, x_2) \,  \psi(0, x_2) \, dx_2\bigg| \le c_1  \, \varepsilon^{\frac{1}{2}}, \quad \phi \in H^2(\Omega^+),  \ \ \psi \in H^1({\Omega}^{+, \varepsilon}),
$$
which has the same proof as the similar inequality in Theorem 4.1 \cite[\S 4, Chapt. 3 ]{Ole-Ios-Sha-1991}.

\section*{Acknowledgments}
The author  thank  for  funding by the  Deutsche Forschungsgemeinschaft (DFG, German Research Foundation) – Project Number 327154368 – SFB 1313.



\end{document}